\documentclass[11pt,reqno,english]{amsart}
\usepackage{amsfonts,amsmath,latexsym,verbatim,amscd,mathrsfs,color,array}

\usepackage[margin=1.2in]{geometry}
\usepackage{amssymb}
\usepackage{pdfsync}
\usepackage{epstopdf}

\newcommand{\ass}{\quad\mbox{as}\quad}

\newcommand{\inn}{{\quad\hbox{in } }}

\newcommand{\ttt}{\tilde }

\newcommand{\LL}{{\tt L}  }
\newcommand{\TT}{{\mathcal T}  }

\newcommand{\nn}{ {\nabla}  }

\newcommand{\pp}{ {\partial} }

\newcommand{\vp}{\varphi}

\newcommand{\RR}{{{\mathbb R}}}

\newcommand{\GG}{ {\mathcal G}}

\newcommand{\FF}{ {\mathcal F}}

\newcommand{\N}{\mathbb{N}}
\newcommand{\C}{{\mathbb C}}
\newcommand{\R} {\mathbb R}
\newcommand{\Z} {\mathbb Z}
\newcommand{\cuad}{{\sqcap\kern-.68em\sqcup}}

\newcommand{\ve}{\varepsilon}

\newcommand{\be}{\begin{equation}}
	\newcommand{\ee}{\end{equation}}

\newcommand{\equ}[1]{(\ref{#1})}

\newcommand{\curl}{\mathop{\rm curl}}

\newtheorem{lemma}{Lemma}[section]
\newtheorem{prop}{Proposition}[section]
\newtheorem{theorem}{Theorem}

\newtheorem{remark}{Remark}[section]
\newcommand{\bremark}{\begin{remark} \em}
	\newcommand{\eremark}{\end{remark} }

\numberwithin{equation}{section}

\begin{document}

	\title[Nearly parallel  helical filaments for 3d Euler]{Nearly parallel helical vortex filaments  in the three-dimensional Euler equations }

	\author[I. Guerra]{Ignacio Guerra*}
	\address{\noindent I.G.: Universidad de Santiago de Chile (USACH) \\ 
		Facultad de Ciencia, Departamento de Matem\'atica y Ciencia de la Computaci\'on \\ 
		Las Sophoras 173, Estaci\'on Central, Santiago, Chile.}

	\email{ignacio.guerra@usach.cl}
	\thanks{* Corresponding Author}
	\author[M. Musso]{Monica Musso}
	\address{\noindent M.M.:  Department of Mathematical Sciences University of Bath,
		Bath BA2 7AY, United Kingdom.}
	\email{m.musso@bath.ac.uk}

	\begin{abstract}
		Klein, Majda, and Damodaran have previously developed a formalized asymptotic motion law describing the evolution of nearly parallel vortex filaments within the framework of the three-dimensional Euler equations for incompressible fluids. In this study, we rigorously justify this model for two  configurations: the central configuration consisting of  regular polygons of $N$ helical-filaments rotating with constant speed, and the central configurations of $N+1$ vortex filaments, where an $N$-polygonal central configuration surrounds a central straight filament.
		
	\end{abstract}
	\maketitle

	\section{Introduction}

	An ideal, incompressible, homogeneous fluid with constant density is described by the Euler equations:
	\be \label{euler0}
	\begin{aligned}
		{\textbf u}_t  + ({\textbf u}\cdot \nn ){\textbf u} &= - \nn p  &&\inn \R^3 \times [0,T) , \\
		{\textrm{div}}\, {\textbf u} &= 0  && \inn \R^3\times [0,T)  .
	\end{aligned}
	\ee
	Here $ {\textbf u} \, : \,  \R^3 \times [0, T ) \to \R^3$ is the velocity field and $p \, : \, \R^3 \times [0, T ) \to \R$
	is the pressure. 
	For a solution ${\textbf u}$ of the Euler equations \equ{euler0}, the vorticity is defined as $ \vec{\omega} = \nabla \times {\textbf u} $. Consequently, $\vec{\omega}$ satisfies the Euler equations in the \emph{vorticity form} 
	\be \label{euler}
	\begin{aligned}
		\vec \omega_t  +
		({\textbf u}\cdot \nn ){\vec \omega}
		&=( \vec \omega \cdot \nn ) {\textbf u}  &&  \inn \R^3\times (0,T), \\
		\quad {\textbf u}  = \curl \vec \psi,\ &
		-\Delta \vec \psi =  \vec \omega  && \inn \R^3\times (0,T).
	\end{aligned}
	\ee
	Vortex filaments 
	are solutions of the Euler equations whose vorticity is concentrated in a small tube near an evolving imaginary smooth curve.
	The mathematical study of the evolution of vortex filaments within the context of classical fluid equations is a challenging problem that traces back to the nineteenth century with the work of Kelvin and Helmholtz. Simplified asymptotic equations have been considered as potential models for describing the asymptotic behavior of small vortex cores.
	Da Rios  \cite{darios}-\cite{levicivita1908}  found that, if  the vorticity
	concentrates smoothly and symmetrically in a small tube of size $\ve >0$ around one smooth curve  parametrized as  $ \gamma(s,t)$ where $s$ designates its arclength parameter, then $\gamma(s,t)$ asymptotically obeys a law of the form
	\[
	\gamma_t =  2\,  |\log\ve| \,   (\gamma_s\times \gamma_{ss})
	\]
	as $\ve \to 0$, or scaling $t= |\log\ve|^{-1}\tau $,
	\be\label{bin1}
	\gamma_\tau =  2\, \kappa \, (\gamma_s\times \gamma_{ss})  .
	\ee
	Here, $\kappa$ is a constant related to the {\em circulation} of the velocity field along the boundary of cross-sections of the filament. See \cite{ricca,majda-bertozzi,jerrard-seis} for a complete discussion on this topic.
	
	\medskip
	Problem \eqref{bin1} has some trivial solutions: the stationary vortex line, the uniformly translating circle (vortex ring), and the translating-rotating helix (helical filament). The "vortex filament conjecture" \cite{bcp,jerrard-seis,jerrard-smets} posits that vorticity initially concentrated around a curve remains close to a curve evolving according to \eqref{bin1}  to leading order for a certain period of time. Jerrard and Seis \cite{jerrard-seis} provided a rigorous derivation of \eqref{bin1} assuming the vorticity remains concentrated around the curve.

	Fraenkel \cite{fraenkel} demonstrated that, in the special case of vortex rings, there exists a family of solutions to \eqref{euler} where the vorticity remains concentrated on a curve solution of \eqref{bin1} for all time. See also \cite{fraenkel-berger, norbury}. Using elliptic singular perturbation techniques,  in \cite{ddmw2} it is proven that there exist solutions  concentrating into one or several polygonally distributed rotating-translating helical filaments, at a certain fixed distance one from the other.  Subsequent constructions can be found  in \cite{cao1, cao2, guerra-musso, jln}.  The dynamics of helical filaments starting from generic initial data have been analyzed in \cite{dlm}. The Biot-Savart law chosen in the above results  gives that the velocity field of the solutions has zero swirl. Global well-posedness and long time behaviour of solutions with helical symmetry whose velocity field does not have vanishing swirl have been obtained in \cite{guo-zhao1, guo-zhao2}.

	\medskip
	In \cite{klein-majda-damodaran}, Klein, Majda and Damodaran have formally derived a
	simplified asymptotic  law of motion for the evolution of several, nearly parallel vortex
	filaments in the context of the  Euler equation \eqref{euler}. 
	It governs 
	the motion of $N \geq 2$ filaments and takes into consideration the pairwise
	interaction between the filaments along with an approximation for
	motion by self-induction. The model is 
	\be \label{KMD1}
	\partial_t X_j = i \, \alpha_j \, \kappa_j \partial_s^2 X_j + 2 i \sum_{k\not= j} \kappa_k {\hspace{-0.5em} X_j - X_k \over |X_j - X_k|^2} , \quad j=1, \ldots , N
	\ee
	where $X_j = X_j (s,t) \in \R^2$ and $s$ is the third coordinate in $\R^3$. The number $\kappa_j$ is related to the circulation strength of the $j$-th filament, while $\alpha_j$ is a constant determined by its vortex core structure. This model is anticipated to be legitimate when the filaments are assumed to have small core section of size $\ve >0$, an amount of deformation of the filaments with respect to
	perfectly straight lines and separation distance between distinct
	filaments with size of the order ${1\over \sqrt{|\log \ve |}}$. (See \cite{klein-majda-damodaran}.)

	\medskip
	System \eqref{KMD1} has been studied for its own (see e.g. \cite{LM,KPV,BM,BFM}), in particular its well-posedness and the possibility of
	colliding filaments. Nevertheless,  the justification
	of the model itself as a limit from a classical fluid mechanics model such as the
	Euler equation has so
	far only been obtained formally through matched asymptotic. The purpose of our paper is to provide this justification for two special configurations satisfying \eqref{KMD1}.

	\medskip
	We start with the observation that, if  $\alpha_j \, \kappa_j = \kappa_0$ for all 
	$j=1, \ldots , N$, then
	the system \eqref{KMD1} is Galilean invariant. This means that if $(X_1, \ldots , X_N)$ is a solution to \eqref{KMD1}, then
	$$
	\tilde X_j (s,t)= e^{-i \kappa_0 \nu^2 t } \, e^{i s \nu } X_j (s\, -\, 2 \, \kappa_0 \,  \nu \, t, t), \quad j=1, \ldots , N
	$$
	is also a solution to \eqref{KMD1} for any $\nu \in \R$, see for instance \cite{KPV}.
	
	Moreover, if $\kappa_j = \kappa$ for all $j=1, \ldots , N$, \eqref{KMD1} has a solution consisting of $N$ points at the vertices of a regular $N$-polygon of radius $r$ rotating with constant speed
	$$
	Z_j (t) = r \, e^{i \, \kappa \,  {N - 1 \over r^2} \, t} \, e^{{2\pi \over N} (j-1) \, i} .
	$$
	This solution corresponds to a central configuration of exactly parallel vortex filaments given by a family of
	straight and exactly parallel lines $(Z_j (t), s)$, see \cite{craig1,craig2}. Normalizing all constants $\alpha_j$ to be equal to $1$, and using 
	the Galilean invariance \eqref{KMD1} it follows that for any $r>0$ and any $\nu \in \R$
	$$
	\tilde Z_j (s,t) = r\,  e^{- i \, \kappa \,  ( \nu^2 - {N-1 \over r^2} )\, t } \, e^{i s \nu} \, e^{{2\pi \over N} (j-1) \, i} 
	$$
	is also a solution to \eqref{KMD1}. A scaling in time allows us to fix the value of $\kappa$, say $\kappa =2$.
	
	Hence the filaments have the form of an $N$-helix
	\be \label{helices}
	\gamma_j (s,t) = \left( \begin{matrix} r  \, e^{-i \, 2 \, ( \nu^2 - {N-1 \over r^2}) \, t }  \, e^{i \, s \, \nu} \, e^{{2\pi \over N} (j-1) \, i} \\
		s \end{matrix} \right) \quad j=1, \ldots , N.
	\ee
	
	These are circular helices of radius $r$. Let $h \in \R$ be such that the  pitch of these helices is $2 \, \pi \, h$, namely
	$$
	\gamma_j (s + 2 \pi h) = \gamma_j (s) + \left( \begin{matrix} 0 \\
		0 \\
		2 \pi h \end{matrix} \right).$$
	Then necessarily  $\nu \, h =1$.

	Observe that for $\nu = \nu_*$
	with
	$
	\nu_*^2 ={N-1 \over r^2},
	$
	we obtain a stationary solution of \eqref{KMD1}  in the form of an $N$-helix
	\begin{equation}\label{stat}
		\gamma_j (s) = \left( \begin{matrix} r  \, e^{i s \nu_*} \, e^{{2\pi \over N} (j-1) \, i} \\
			s \end{matrix} \right) \quad j=1, \ldots , N.
	\end{equation}
	If $\nu =0$, we obtain a solution of \eqref{KMD1} constituted by $N$ straight lines.

	\medskip
	The goal of this paper is to construct a family of  solutions to \eqref{euler} whose vortex-set is close, as $\ve $ tends
	to zero, to rescaled versions of the helices \eqref{helices}. 
	The solutions we construct are invariant under helical symmetry, in the sense that
	$$
	\vec \omega (\, S_{-\rho} \, x , t) = R_\rho \, \vec \omega \, (x,t) \quad \forall \rho $$
	where
	$$
	R_\rho = \left( \begin{matrix} \cos \rho & -\sin \rho &0 \\
		\sin \rho & \cos \rho &0 \\
		0&0&1
	\end{matrix} \right),  \quad S_\rho = R_\rho + \left( \begin{matrix} 0 \\
		0 \\
		h \rho\end{matrix} \right).
	$$
	The associated velocity field ${\bf u}$ is orthogonal to the symmetry lines of the action generated by $S_\rho$. The vorticity rotates at a constant speed, which is determined by the radius of the helices and their pitch, along with lower-order terms as $\ve \to 0$.
	
	\medskip
	\medskip
	
	We prove the following result.
	
	\begin{theorem}\label{teo2} 
		Let
		$$
		r>0, \quad h \not=0
		$$   
		be given.
		Let $\gamma_j$ be the helices parametrized by \eqref{helices} with $\nu= {1\over h}$, for $j=1,\ldots,N$. Then
		there exist $\alpha_\ve >0$ and  a smooth  solution $\vec \omega_\ve (x,t)$ to $\equ{euler}$ such that
		\[
		\vec\omega_\ve (x,t) -  8 \, \pi \, \sum_{j=1}^N  \delta_{\gamma_j^\ve }{\textbf t}_{\gamma_j^\ve } \rightharpoonup  0 \ass \ve \to 0,
		\]
		in sense of distributions. Here 
		\begin{align*}
			\gamma_j^\ve (s,t) &=\,  \left( \begin{matrix}  {1\over \sqrt{|\log \ve |}}  \, r  \, e^{-i \,  \alpha_\ve \, t} \, e^{i \, {s \over h} } \, \, e^{{2\pi \over N}\,  (j-1) \, i} \\
				s \end{matrix} \right), \\ 
			& \\
			\alpha_\ve &=  2 \left({1\over h^2} - {N-1 \over r^2} \right) + O \left( {\log |\log \ve | \over |\log \ve |} \right) \quad {\mbox {as}} \quad \ve \to 0,
		\end{align*}
		and ${\textbf t}_{\gamma_j^\ve }$ is the tangent vector to $\gamma_j^\ve.$
	\end{theorem}

	\medskip
	Our result encompasses stationary solutions where vorticity is concentrated along $N$ nearly parallel helical filaments, corresponding to the stationary solution of \eqref{KMD1} given by \eqref{stat}. In this case, the pitch $h$ is related to the radius $r$ by
	$$
	r^2 = (N-1) \, h^2  + O \left( {\log |\log \ve | \over |\log \ve |} \right) \quad {\mbox {as}} \quad \ve \to 0.
	$$

	\medskip
	Another configuration solving \eqref{KMD1} consists of $N+1$ point vortices, setting $N$ of them
	with identical circulation at each vertex of a regular $N$-polygon and the $N+1$ vortex of
	arbitrary strength at the center of vorticity, namely the origin
	$$
	Z_{0} (t) = 0 , \quad Z_{j} (t) = r \, e^{-i \, \kappa \,  {N + 1 \over r^2} \,  t} \, e^{{2\pi \over N} (j-1) \, i}.
	$$
	Arguing as before and choosing again $\kappa =2$, one obtains a solution of \eqref{KMD1}  in the form of  $N$ circular helices surrounding a straight line
	\be \label{helicessharp}
	\gamma_{0,\sharp} (s,t) = \left( \begin{matrix} 0 \\ 0\\
		s \end{matrix} \right)  , \quad \gamma_{j,\sharp} (s,t) = \left( \begin{matrix} r  \,  e^{- 2 \, ( {1\over h^2} - {N+1 \over r^2}) \, i \, t } \, e^{i \, {s \over h}  } \, e^{{2\pi \over N} (j-1) \, i} \\
		s \end{matrix} \right) \quad j=1, \ldots , N.
	\ee
	
	\medskip
	\medskip
	
	We have a result analogous to the one in Theorem \ref{teo2} also for this configuration
	
	\begin{theorem}\label{teo} 
		Let
		$$
		r>0, \quad h \not=0
		$$   
		be given.
		Let $\gamma_{0,\sharp}$, $\gamma_{j, \sharp}$, for $j=1,\ldots,N$, be parametrized by \eqref{helicessharp}. Then
		there exist $\alpha_\ve >0$ and  a smooth  solution $\vec \omega_\ve (x,t)$ to $\equ{euler}$ such that
		\[
		\vec\omega_\ve (x, t) -  8 \, \pi \,\sum_{j=1}^N  \delta_{\gamma_{j,\sharp}^\ve }{\textbf t}_{\gamma_{j,\sharp}^\ve } \rightharpoonup  0 \ass \ve \to 0,
		\]
		in sense of distributions. Here 
		\begin{align*}
			\gamma_{j,\sharp}^\ve (s,t) &=\,  \left( \begin{matrix}  {1\over \sqrt{|\log \ve |}}  \, r  \, e^{-i \,  \alpha_\ve \, t} \, e^{i \, {s \over h} } \, \, e^{{2\pi \over N}\,  (j-1) \, i} \\
				s \end{matrix} \right), \\ 
			& \\
			\alpha_\ve &=  2 \left({1\over h^2} - {N+1 \over r^2} \right) + O \left( {\log |\log \ve | \over |\log \ve |} \right) \quad {\mbox {as}} \quad \ve \to 0,
		\end{align*}
		and ${\textbf t}_{\gamma_{j,\sharp}^\ve }$ is the tangent vector to $\gamma_{j,\sharp}^\ve.$
		
	\end{theorem}

	We observe that the deformation magnitude of each helix $\gamma_j^\ve$ in Theorem \ref{teo2}, or of each helix  $\gamma_{j, \sharp}^\ve$ in Theorem \ref{teo}, relative to the perfectly straight line $ \left( \begin{matrix} 0 \\ 0 \\ s \end{matrix} \right) $ is proportional to ${1\over \sqrt{|\log \ve|}}$. Additionally, the relative distance between helices is also of the order of ${1\over \sqrt{|\log \ve|}}$. This description aligns with the model proposed by Klein, Majda, and Damodaran \cite{klein-majda-damodaran}.

	\medskip
	Multiple helical filaments for the Euler equations have been studied in previous papers. In \cite{ddmw2} and \cite{dlm}, the helical filaments are separated by a distance independent of $\ve$. Therefore, those solutions do not fit the model for nearly parallel vortex filaments as described in \cite{klein-majda-damodaran}. In \cite{guerra-musso}, we construct solutions to \eqref{euler} with vorticity concentrated around several  helical filaments. Those helical filaments rotate with constant speed and cluster around a straight line of the form $ \left( \begin{matrix} 0 \\ 0 \\ r_0 + s \end{matrix} \right) $, with $r_0>0$. In that construction \cite{guerra-musso}, the relative distance among the filaments is of order ${1\over |\log \ve|}$ and the asymptotic motion law for the centers of these helices differs from \eqref{KMD1}, leading to solutions distinct from those in Theorems \ref{teo2} and \ref{teo}.

	\medskip
	As far as we are aware, a  justification in full generality of the Klein-Majda-Damodaran model remains an open question for the Euler equations. Specifically, given a solution to \eqref{KMD1}, Does there exist any solution to \eqref{euler} with initial vorticity concentrated around the curves $ \left( \begin{matrix} X_1(s,0) \\  s \end{matrix} \right), \ldots , \left( \begin{matrix} X_N(s,0) \\  s \end{matrix} \right)$, such that the vorticity remains concentrated in small tubes around an evolving solution to \eqref{KMD1}?.

	\medskip
	System \eqref{KMD1} has been rigorously derived as a limit from PDE models related to fluid mechanics other than the Euler equations. This derivation was done in \cite{jerrard-smets-gp} for the Gross-Pitaevskii equations in three dimensions, which appear in various contexts such as Bose-Einstein condensate theory, nonlinear optics, and superfluidity.
	
	Associated with the helix, there exist solutions to the Gross-Pitaevskii equation. Traveling wave solutions with small speed were constructed in \cite{chiron}. Interacting helical traveling waves for the Gross-Pitaevskii equations have been considered in \cite{ddmr1}. The corresponding traveling single-helix solutions to the Schrödinger map equation were proved to exist in \cite{wei-yang}, and for the Ginzburg-Landau equations for superconductivity in \cite{contreras-jerrard, ddmr}.
	
	\medskip
	Our construction leverages a screw-driving symmetry invariance present in the Euler equations, as observed in \cite{dutrifoy} and \cite{ettinger-titi}. For more recent results, see \cite{bnl,jln,velasco}. According to the analysis in these references, a vector field of the form
	\be \label{vector-vorticity}
	\begin{aligned}
		\vec\omega(x,\tau ) = \, w( Q_{-\frac {x_3}h} x', |\log \ve |^{-1} \tau )
		\left (\begin{matrix} \, \frac{1}{h}  \,  Q_{\frac \pi 2} x' , \, 1 \end{matrix} \right) , \qquad x'= (x_1,x_2),
	\end{aligned}
	\ee
	solves \eqref{euler} if the scalar function $w(x',t)$ satisfies the transport equation
	$$
	\begin{aligned}
		|\log \ve |	\partial_t w + \nabla^\perp \psi \cdot \nabla w &=0 && \text{in} \quad \mathbb{R}^2 \times (0, T) \\-\nabla \cdot (K \nabla \psi) &= w && \text{in} \quad \mathbb{R}^2 \times (0, T),
	\end{aligned}
	$$
	where $(a,b)^\perp = (b,-a)$, $h>0$ is a given number  and
	$K(x_1 , x_2)$ is the matrix
	\begin{equation}
		\label{defK-0}
		K(x_1 , x_2 ) = \frac{1}{h^2 + x_1^2 + x_2^2}
		\left(
		\begin{matrix}
			h^2 + x_2^2 & -x_1 x_2 \\
			-x_1 x_2 & h^2 + x_1^2
		\end{matrix}
		\right).
	\end{equation}
	In \eqref{vector-vorticity}, $Q_\theta$ stands for the rotation matrix 
	\begin{equation}\label{rotation-matrix}
		Q_\theta = \left(\begin{matrix} 
			\cos \theta&  -\sin \theta \\
			\sin \theta & \cos \theta
		\end{matrix} \right).
	\end{equation}
	Rotating solutions to the problem with constant speed $\alpha$ have the form
	$$
	w (x', \tau) = w \left( Q_{ \alpha \tau } x' \right), \quad \psi (x',\tau) = \psi \left( Q_{ \alpha \tau } x' \right). 
	$$
	They can be obtained by solving the elliptic equation
	\be \label{P}
	\nabla \cdot ( K \nabla \psi) + F(\psi -{\alpha \over 2} \, |\log \ve | \, |x|^2)=0 \quad \text{in} \quad \mathbb{R}^2,
	\ee
	for any function $F \in C^1.$ In this case, $\psi$ is the stream function of the fluid and the vorticity-strength  is given by $\omega = F(\psi -{\alpha \over 2} \, |\log \ve | \, |x|^2 )$. See \cite{ddmw2} for detailed derivation of the formulation \eqref{P}.
	
	A solution with helical symmetry whose vorticity is $\ve$-concentrated around one helical filament corresponds to a solution of \eqref{P} such that
	$$
	\omega - c \delta_P = F(\psi -{\alpha \over 2} \, |\log \ve | \, |x|^2) - c \delta_P \rightharpoonup 0  \ass \ve \to 0,
	$$
	where $\delta_P$ is the Dirac delta centered at the point $P \in \R^2$ and $c$ is a constant.
	
	\medskip
	In the proof of Theorem \ref{teo2}, we proceed as follows: We construct the stream function $\psi$ for several helical filaments as a superposition of $\ve$-regularized Green's functions for the elliptic operator
	$\nabla \cdot (K \nabla \psi) $ centered at points of the form
	$$
	P_j = {r \over \sqrt{|\log \ve |}} \, e^{{2\pi \over N} \, (j-1) \, i}, \quad j=1, \ldots , N.
	$$
	This construction is detailed in Section \S \ref{app}, following a discussion on the operator $\nabla \cdot (K \nabla \psi) $ in Section \S \ref{sub1}. Next in Section \S \ref{sec4} we choose an appropriate non-linearity $F$ and construct an approximate solution to \eqref{P} with the property that
	$$
	F(\psi -{\alpha \over 2} \, |\log \ve | \, |x|^2) - 8\pi \, \sum_{j=1}^N \delta_{P_j} \rightharpoonup 0  \ass \ve \to 0.
	$$
	The constant $8\pi$  is related to our choice of the profile $F$. Once the approximate solution is built, we devise a strategy to prove that a small perturbation yields an actual solution to the problem. Our strategy involves setting up an inner-outer gluing scheme in Section \S \ref{quattro}, that we completely solve in Section \S \ref{sette}. It is at this point that the choice of $\alpha$ to be  $2 \left({1\over h^2} - {N-1 \over r^2} \right) + o(1)$ as $\ve \to 0$ is required, as in the statement of Theorem \ref{teo2} and in accordance with the discussion leading to \eqref{helices}. For clarity, we defer some technical results to Sections \S \ref{AppeA} and \S \ref{AppeB}.
	Since the proof of Theorem \ref{teo} is a relatively straightforward adaptation of the proof of Theorem \ref{teo2}, we will omit it.

	\section{Preliminaries on the operator $\nabla \cdot ( K \, \nabla \, )$}
	\label{sub1}

	Let   $L$ denote the elliptic operator
	\be \label{defL}
	\begin{aligned}
		L&:=-\nabla \cdot ( K \, \nabla \, ),
	\end{aligned}
	\ee
	where $K$ is defined in \eqref{defK-0}.
	The main purpose of this section is to find an approximate stream function $\psi$ with the property
	$$
	L ( \psi) \sim c \, \delta_P,
	$$
	where $\delta_P$ is the Dirac delta centered at the point $P \in \R^2$ and $c$ is a constant.
	
	We achieve this in several steps. First, we analyze the operator $L$ and prove that it is a small perturbation of the standard Laplace operator, when considered in a small neighborhood of a given point $P \in \R^2$ and after performing an appropriate change of variable. We take a point of the form
	$P= (R,0)$, for some $R>0$ and the appropriate change of variable  given by  
	\be \label{matrixdef}
	x-(R,0) = M \, z, \quad M = \left( \begin{matrix} 
		{  h\over  \sqrt{h^2+  R^2} } & 0 \\
		0 
		& 1
	\end{matrix} \right).
	\ee
	In this Section, we choose a specific regularization of the Green's function and describe its asymptotic behavior in a fixed, small neighborhood of $P$. This yields an approximate stream function, initially defined only locally around $P$.

	\medskip
	We proceed assuming that $h\not= 0$ and $P=(R,0)$ are given.
	For a function $\psi = \psi(x)$, we use the notation $\Psi$ to express $\psi$ after the change of variable  \eqref{matrixdef}
	$$
	\Psi (z) = \psi \left( Mz + (R,0) \right).
	$$
	Next proposition shows that
	the operator $L$ is a perturbation of the Laplace operator when expressed in the $z$-variable and evaluated around the point $(R,0)$. 
	
	\begin{prop}\label{prop01}
		It holds
		\be \label{Lz}
		L \psi = L_0 \Psi , \quad {\mbox {where}} \quad L_0 :=  \Delta_z + B,
		\ee
		and 
		$$
		\begin{aligned}
			B&=-\left(   
			{2R h \over  (h^2+ R ^2)^{3/2}} \, z_1 + O(|z|^2 )\right) \pp_{z_1 z_1} + O(|z|^2 ) \pp_{z_2 z_2} \\
			&-\left( 
			{2 R  \over h\sqrt{h^2+ R^2} } z_2 + O(|z|^2) \right) \pp_{z_1 z_2}
			\\
			- \left( {R \over  h\sqrt{h^2+ R^2} }\right. &\left.\left({2h^2\over h^2+ R^2} +1 \right)  + O(|z|) \right) \pp_{z_1 }  
			- \left(\frac{z_2}{h^2+ R^2} \left({2 h^2\over h^2+  R^2} +1 \right)+O(|z|^2)\right) \pp_{z_2 }.
		\end{aligned}
		$$
		where $|z|\leq \delta$, for any $\delta >0$ small.

	\end{prop}
	\begin{proof}
		Using the change of variables \eqref{matrixdef}, one gets \eqref{Lz} with
		\be\label{Bexpr}
		\begin{aligned}
			B&=\left( {h^2(R^2-|x|^2)+z_2^2(h^2+R^2) \over (h^2+ r^2)h^2}\right) \pp_{z_1 z_1} 
			+ {1 \over (h^2+ |x|^2)}\left(\left(z_1\frac{h}{\sqrt{h^2+ R^2}}+R\right)^2-|x|^2\right) \pp_{z_2 z_2} \\
			&-2{\sqrt{h^2+R^2} \over h (h^2+ |x|^2)}z_2\left( z_1
			{h \over \sqrt{h^2+ R^2}} +R \right) \pp_{z_1 z_2}\\
			&- \frac{z_1(h^2+R^2)+Rh\sqrt{h^2+ R^2}}{h^2(h^2+|x|^2)} \left({2 h^2\over h^2+ |x|^2} +1 \right)  \pp_{z_1 }
			- \frac{z_2}{h^2+|x|^2} \left({2 h^2\over h^2+ |x|^2} +1 \right) \pp_{z_2 }.
		\end{aligned}
		\ee
		Here $|x|^2=R^2 + 2 {R h \over \sqrt{h^2 + R^2}}  z_1 + {h^2 \over h^2 + R^2} z_1^2 + z_2^2.$
		Detailed computations can be found in the proof of Proposition 2.1 in \cite{guerra-musso}, taking $a=R$ and $b=0$.
	\end{proof}

	A vorticity concentrated around the point $(R,0)$ has the form 
	$$
	- L_0 \Psi =  c\, \delta_{(R,0)},
	$$
	where $\delta_P$ denotes the Dirac delta at the point $P$, and $c>0$ is a constant.
	Next Proposition introduces a family of regular functions which resemble a multiple of the Green function for the operator $L_0$. This is constructed upon
	a radial solution of the Liouville equation
	\be \label{liouville}
	\Delta u + e^u = 0 \quad \inn \R^2, \quad \int_{\R^2} e^u \, = 8\pi < \infty.
	\ee
	All solutions to \eqref{liouville} that are radially symmetric with respect to the origin are given by 
	\be \label{defGamma}
	\Gamma_{\mu \ve} (z) - 2 \log \ve \mu , \quad {\mbox {where}} \quad \Gamma_{\ve  \mu } (z) = \log {8 \over (\ve^2 \mu^2 + |z|^2 )^2}
	\ee
	for any value of the constants $\ve$ and $\mu>0$. Indeed we have
	$$
	-\Delta \Gamma_{\ve  \mu }=\ve^2 \mu^2 e^{\Gamma_{\ve \mu}} = {1\over \ve^2 \mu^2} U \left( {z \over \ve \mu} \right)  \quad {\mbox {with}} \quad U(y) = \frac8{(1+ |y|^2)^2}.  
	$$
	We define
	\be\label{defGamma0}
	\Gamma (y) = \log {8 \over (1 + |y|^2 )^2}, \quad {\mbox {so that}} \quad \Delta_y \Gamma + e^\Gamma =0.
	\ee
	A direct computation shows that
	$$
	-\Delta \Gamma_{\ve  \mu } \rightharpoonup 8 \pi \delta_0 , \quad \ass \ve \mu \to 0.
	$$
	
	\begin{prop}\label{prop02}
		For any $\ve >0$ and $\mu >0$, we define the function
		\be \label{defpsi2}\begin{aligned}
			\Psi_{ \ve  \mu } (z) &= \Gamma_{\ve \mu} (z) \left( 1+c_1 z_1 + c_2 |z|^2\right) +   {4 R^3 \over h(h^2+ R^2)^{3\over 2}}  H_{1\ve} (z) ,\\
			c_1 & ={1\over 2}  {R h \over (h^2+R^2)^{3\over 2} }\\
			c_2 &= {R^2 \over 8 (h^2+ R^2)^2 } \left({2 h^2\over h^2+ R^2} +1  \right)
		\end{aligned}
		\ee
		where $\Gamma_{\ve \mu}$ is given in \eqref{defGamma}
		and $H_{1\ve} $ has the form $H_{1\ve}  (z ) := h_1 (|z|) \cos 3 \theta$, with $z=|z|e^{i\theta}$ and it solves
		\be\label{H1}
		\Delta_z H_{1\ve}  + {{\mbox{Re} } (z^3) \over (\ve^2 \mu^2 + |z|^2 )^2} = 0.
		\ee
		Write
		\begin{equation}\label{ee10}
			L_0 (\Psi_{\ve \mu} ) (z)= \Delta \Gamma_{\ve \mu}  + {4 R (3h^2+ R^2) \over h(h^2+ R^2)^{3\over 2} }  { \ve^2 \mu^2 \, z_1 \over (\ve^2 \mu^2 + |z|^2 )^2}  
			+ E(z)  .
		\end{equation}
		Then for any $C>0$, $\delta >0$, there exists $\bar \ve >0$ such that, for all $\ve \in (0, \bar \ve)$
		$$
		|h_1 (z) | \leq C |z|, \quad |E (z)| \leq C
		$$
		uniformly for all $|z| \leq \delta$, $0 < \mu \leq \delta$, $0<R\leq \delta$.
	\end{prop}
	
	\begin{proof}
		The proof uses Proposition \ref{prop01} and is analogous to Proposition 2.2 in \cite{guerra-musso}. Here we make a sketch of the proof.  
		Let $\delta >0$ be small. Using the explicit form of $\Gamma_{\ve\mu}$, we can compute
		\begin{align*}
			B [ \Gamma_{\ve  \mu }] &= - {2 Rh \over (h^2+R^2)^{3/2} } z_1 \pp_{z_1 z_1} \Gamma_{\ve  \mu }  -  {2 R \over h\sqrt{h^2+ R^2}}  z_2 \pp_{z_1 z_2} \Gamma_{\ve  \mu }  \\
			&- {R \over h\sqrt{h^2+ R^2}} \left(1+ {2h^2 \over h^2+ R^2} \right) \pp_{z_1} \Gamma_{\ve  \mu }+ E_1 \\
			&= {4h R \over (h^2+ R^2)^{3\over 2} } {z_1 \over \ve^2 \mu^2 + |z|^2} + {4 R^3 \over h (h^2+ R^2)^{3\over 2}} {{\mbox {Re} } (z^3 ) \over (\ve^2 \mu^2 + |z|^2)^2} \\
			&+ {4 R (4h^2+ R^2) \over h (h^2+ R^2)^{3\over 2} }  { \ve^2 \mu^2 \, z_1 \over (\ve^2 \mu^2 + |z|^2 )^2}+E_1.
		\end{align*}
		where $E_1$ is a smooth  function, uniformly bounded for $\ve \mu $ small, in a bounded region for $|z|<\delta $.
		
		Now we modify the function $\Gamma_{\ve \mu}$ to eliminate the first term of the above error. We define
		\be \label{defc}
		\begin{aligned} 
			\Psi_{1} (z) &= \Gamma_{\ve \mu} (z) \left( 1+ c_1 z_1  \right), \quad {\mbox {with}} \\
			c_1 & ={1\over 2}  {R h \over (h^2+R^2)^{3\over 2} }\\
		\end{aligned}
		\ee
		Since
		\[
		\Delta_z (c_1 z_1 \Gamma_{\ve \mu} ) = - 8 \, c_1 \, {z_1 \over \ve^2 \mu^2 + |z|^2} - 8\,  c_1 \, {\ve^2 \mu^2 z_1 \over (\ve^2 \mu^2 + |z|^2 )^2}, 
		\]
		with this choice of  $c_1$,
		we get
		\begin{align*}
			L_0 (\Psi_1 ) (z)&= \Delta \Gamma_{\ve \mu} +  {4 R^3 \over h(h^2+ R^2)^{3\over 2}} {{\mbox {Re} } (z^3 ) \over (\ve^2 \mu^2 + |z|^2)^2} + {4 R (3h^2+ R^2) \over (h^2+ R^2)^{3\over 2} }  { \ve^2 \mu^2 \, z_1 \over (\ve^2 \mu^2 + |z|^2 )^2}\\
			& 
			+ c_1 B[z_1 \Gamma_{\ve \mu} ] 
			+ E_1
		\end{align*}
		where $E_1$ is an explicit function, which smooth in the variable $z $ and uniformly bounded, as $\ve \mu \to 0$.
		
		\medskip
		Similarly as above we compute
		\begin{align*}
			c_1 B[z_1 \Gamma_{\ve \mu} ] &
			= - {R^2 \over 2(h^2+ R^2)^2 } \left({2h^2\over h^2+ R^2} +1 \right)  \Gamma_{\ve \mu}
			+ E_2 
		\end{align*}
		where $E_2$ is another explicit function,  smooth in the variable $z $ and uniformly bounded, as $\ve \mu \to 0$.
		
		\medskip
		Combining these computations we obtain that the function $\psi_1$ introduced in \eqref{defc} satisfies
		\begin{align*}
			L_0 (\Psi_1 ) (z)&= \Delta \Gamma_{\ve \mu} +  {4 R^3 \over h(h^2+ R^2)^{3\over 2}} {{\mbox {Re} } (z^3 ) \over (\ve^2 \mu^2 + |z|^2)^2} + {4 R (3h^2+ R^2) \over h(h^2+ R^2)^{3\over 2} }  { \ve^2 \mu^2 \, z_1 \over (\ve^2 \mu^2 + |z|^2 )^2}\\
			&- {R^2 \over 2(h^2+ R^2)^2 } \left({2h^2\over h^2+ R^2} +1 \right)  
			\Gamma_{\ve \mu} + E_1 + E_2 ,
		\end{align*}
		where $E_1$ and  $E_2$ are explicit functions, smooth in the variable $z $ and uniformly bounded, as $\ve \mu \to 0$.

		\medskip
		Now we introduce a further modification to  $\psi_{1}$ to eliminate the two terms
		\[
		- {R^2 \over 2(h^2+ R^2)^2 } \left( {2h^2\over h^2+ R^2} +1    \right) \Gamma_{\ve \mu} \quad {\mbox {and}} \quad  {4 R^3 \over h(h^2+ R^2)^{3\over 2}} {{\mbox {Re} } (z^3 ) \over (\ve^2 \mu^2 + |z|^2)^2}.
		\]
		For the first one, we observe that
		\begin{align*}
			\Delta (c_2 |z|^2 \Gamma_{\ve \mu} ) &  - {R^2 \over 2(h^2+ R^2)^2 } \left({2 h^2\over h^2+ R^2} +1  \right) \Gamma_{\ve \mu} \\
			&= 2 c_2 z \cdot \nabla \Gamma_{\ve \mu} + c_2 |z|^2 \Delta \Gamma_{\ve \mu}
		\end{align*}
		if we choose $c_2$ as
		\[
		c_2  = {R^2 \over 8 (h^2+ R^2)^2 } \left({2 h^2\over h^2+ R^2} +1  \right).
		\]
		To correct the second term, we introduce $H_1$ as in \eqref{H1}. This concludes the proof.

	\end{proof}

	\section{The approximate stream function}\label{app}

	For the rest of the paper, we assume we are given two  numbers
	$$
	r>0 {\mbox { and }} h\not= 0
	$$
	as  in the statement of Theorem \ref{teo2}.

	In Subsection \ref{sub2}, we introduce a combination of several copies of the approximate stream functions constructed in Section \ref{sub1}, after we  centre them at the vertices $P_j$ of a regular polygon of radius ${r \over \sqrt{|\log \ve |}}$. In Subsection \ref{sub3}, we multiply this sum of approximate stream functions by an appropriate cut-off function and extend the definition to the whole $\R^2$, adding a correction term. Further adjustments are then required to control the effect of this correction on the local profile of the approximate stream function around each $P_j$. This final step is detailed in Subsection \ref{sub4}.

	\subsection{The approximate stream function near the vortices.}\label{sub2}

	Let $N \in \N$. For any $\ve >0$ small,  we let
	$P_j$ be the points arranged along the vertices of a regular polygon of $N$ sides
	\begin{equation}\label{points}
		\begin{aligned}
			P_j = & R_\ve  \,  Q_j (1,0) ,\quad R_\ve = {r \over \sqrt{|\log \ve|}}, \quad   j=1, \ldots , N ,  \\
		\end{aligned}
	\end{equation}
	where 
	\be \label{defQj}
	Q_j=Q_{\theta_j} ,\quad  \theta_j = {2\pi \, \over N} \, (j-1)  \quad j=1, \ldots , N.
	\ee
	Here $Q_\theta$ is the matrix of rotation of angle $\theta$, as defined in \eqref{rotation-matrix}.

	The rotational speed $\alpha$ will be chosen in terms of   $r$, $h$ and $\ve$. For the moment we assume it is uniformly bounded as $\ve \to 0$, namely  there exists $\alpha_0 >0$ independent of $\ve$ such that
	\begin{equation}\label{alpha}
		\alpha = \alpha_\ve  (r,h) ,  \quad 	 |\alpha | \leq \alpha_0
	\end{equation}
	for all $\ve >0$ small.
	
	\medskip

	For any $\ve >0$ and $\mu >0$, let $\Psi_{ \ve \mu}$ be given by \eqref{defpsi2}. We define
	$$
	\psi_1(x)=\Psi_{\ve \mu} \left(M^{-1}(x-P_1)\right), \quad P_1 = (R_\ve ,0),
	$$
	and 
	\be \label{psi0}
	\psi_0(x)=\sum_{j=1}^N \,  \,  \,  \psi_j (x), 
	\quad {\mbox {where}} \quad \psi_j(x)=\psi_1(Q_j^{-1}x).
	\ee
	We recall that the definition of the matrix $M$ is given in \eqref{matrixdef}.
	The function $\psi_0$ will be the principal part of the solution we are constructing. Its definition depends on $\ve$ and $\mu$. We will make the following assumption on $\mu$: 
	We assume that
	\be \label{defR}
	\begin{aligned}
		\delta & \log |\log \ve | < |\log \mu |< \delta^{-1} \log |\log \ve |,
	\end{aligned}
	\ee
	for some $\delta \in (0,1)$  fixed, independent of $\ve $. In particular, $\ve \mu \to 0$ as $\ve \to 0$.

	\medskip
	The function $\psi_0$ is clearly invariant under rotation of ${2\pi \over N},$ that is 
	$$
	\psi_0 (Q_{2\pi \over N} x ) = \psi_0 (x), \quad \forall x \in \R^2.
	$$
	We define
	\be \label{Psi0}
	\Psi_0 (z) = \psi_0 (P_1 + Mz ), \quad x= P_1 + Mz.
	\ee
	A direct computation gives
	$$
	\Psi_0(z)=\Psi_{ \ve \mu}(z)+\sum_{j=2}^N \Psi_{ \ve \mu}(M_j^{-1}(P_1-P_j+Mz))$$
	where
	\be \label{defMj}
	M_j=Q_jM.
	\ee
	If $j=1$, the matrix $M_1$ coincides with the matrix $M$ in \eqref{matrixdef}.
	We also note that 
	\begin{align*}
		\psi_0(x)=\Psi_0(z)=
		&\sum_{j=1}^N \Psi_{ \ve \mu}(M_j^{-1}(P_1-P_j+Mz)), \quad z=M^{-1}(x-P_1).
	\end{align*}
	For $i \not= 1$, by taking $\hat z=M_i^{-1}(x-P_i)$, we have 
	$M_i \hat z+P_i=M z+P_1,$ and consequently we can write $\psi_0(x)=\sum_{j=1}^N \Psi_{ \ve \mu}(M_j^{-1}(P_i-P_j+M_i\hat z))$ so that
	$$
	\psi_0(x)
	=\Psi_{ \ve \mu}(\hat z)+\sum_{j\neq i }^N \Psi_{ \ve \mu}(M_j^{-1}(P_i-P_j+M_i\hat z)), 
	$$
	where $\hat z=M_i^{-1}(x-P_i)$.
	
	\medskip
	Another property of the function $\Psi_0 (z)$ in \eqref{Psi0} is its even-ness with respect to $z_2$.
	\begin{prop}\label{even-ness}
		We have that \begin{eqnarray}\label{symmetryPsi} \Psi_0(z_1,z_2)=\Psi_0(z_1,-z_2)\end{eqnarray}
		for $\Psi_0$ defined as in \eqref{Psi0}.
	\end{prop}
	
	\begin{proof}  See Appendix \ref{AppeA}. \end{proof}

	\subsection{Global definition for the approximate stream function.}\label{sub3}

	The definition of the stream function of $N$-helical filaments as given by \eqref{psi0} looks at main order as a superposition of stream functions $\psi_j$ associated to each helical filament. This is a good expression for the stream function in regions close to the vortices $P_j$. Under our assumption \eqref{points}, the points $P_j$ are approaching the origin $(0,0)$ at a rate  $\sqrt{|\log \ve |}^{-1}$, when $\ve \to 0$. It is thus natural to multiply $\psi_0$ in \eqref{psi0} by a cut-off function 
	$$
	\eta_0 (|x|) 
	\sum_{j=1}^N \,  \,  \,  \psi_j (x) 
	$$
	where $\eta_0 $ a fixed smooth function such that
	\begin{equation}\label{defeta}
		\eta_0(s) = 1 , \quad {\mbox {for}} \quad s \leq {1 \over 2}, \quad \eta_0(s) = 0  , \quad {\mbox {for}} \quad s \geq 1.
	\end{equation}

	\medskip
	\begin{remark}
		A key property of the operator $L$ is that
		it is rotational invariant: for any $\theta$
		\be \label{rotinv}
		(L\psi )(x)=(L \tilde \psi) (Q_\theta x),\quad \mbox{where} \quad \psi (x) = \tilde \psi (Q_\theta x).
		\ee
		This property  gives an analogous version of Propositions \ref{prop01} and \ref{prop02} when the function is expressed with respect to the change of variables around each $P_j$ given by 
		\be \label{changej}
		x- P_j= M_jz=Q_j\,M\, z,
		\ee
		instead of \eqref{matrixdef}.
	\end{remark}
	
	\begin{prop}\label{24}
		Let $j$ be fixed in $\{ 1, \ldots , N\}$ as defined in \eqref{psi0}. Then
		$$
		L (\psi_j ) (x)= L_0 (\Psi_{\ve \mu} ) ( M_j^{-1} (x-P_j))
		$$
		where $L_0$ is as in \eqref{Lz}. Besides formula \eqref{ee10} is still valid using the change of variables  \eqref{changej} instead of \eqref{matrixdef}.
	\end{prop}
	\begin{proof}  We use \eqref{rotinv}, \eqref{points} and \eqref{psi0} to ger
		\begin{align*}
			L(\psi_j) (x) &= L(\psi_1) (Q_j^{-1} x ) = (L_0 \Psi_{\ve \mu} ) ( M^{-1}
			(Q_j^{-1} x - P_1)) \\
			&=
			(L_0 \Psi_{\ve \mu} ) ( M_j^{-1}
			( x - P_j)).
		\end{align*}
	\end{proof} 
	
	With the result of the above Proposition at hand, using \eqref{ee10} we get that
	\begin{align*}
		L\left( \eta_0 \sum_{j=1}^N \,  \, 
		\psi_j  \right)&=
		\eta_0   \, \,   \left( \Delta \Gamma_{\ve \mu}  + {4 R (3h^2+ R^2) \over h(h^2+ R^2)^{3\over 2} }  { \ve^2 \mu^2 \, z_1 \over (\ve^2 \mu^2 + |z|^2 )^2}  \right) \\
		&+\eta_0\sum_{j=2}^N \,\, \,   \left( \frac{8\ve^2 \mu^2}{(\ve^2 \mu^2+|M_j^{-1}Mz-M_j^{-1}(P_j-P_1)|^2 )^2} \right. \\
		&+ \left. {4 R (3h^2+ R^2) \over h(h^2+ R^2)^{3\over 2} }  { \ve^2 \mu^2 \, [M_j^{-1}Mz-M_j^{-1}(P_j-P_1)]_1 \over (\ve^2 \mu^2 + |M_j^{-1}M z-M_j^{-1}(P_j-P_1)|^2 )^2}  \right) \\
		&+g(x) , \quad {\mbox {with}} \quad z=M^{-1}(x-P_1)
	\end{align*}
	where $M_j$ is the matrix defined in \eqref{defMj}. The function $g$ is
	$$
	g(x)= \eta_0 \sum_{j=1}^N \, 
	\, \,   E(M_j^{-1}(x-P_j))  +
	\sum_{j=1}^N 
	\left[  \, L( \eta_{0} \psi_j ) - \eta_0 L(\psi_j)   \right],
	$$
	where $E$ is as in \eqref{ee10}. It is easy to see that
	$g$ has compact support and satisfies
	$$
	\| g (x) \|_{L^\infty (\R^2 ) } \leq C
	$$
	for some positive constant $C$.
	\medskip
	
	Note that $g$ is invariant under rotation of angle ${2\pi \over N}$. Indeed more generally, we have
	\be \label{dihege}
	g(x')=g(x)\quad\mbox{where}\quad x'=Q_ix,\quad\mbox{for any $i$,}
	\ee
	where $Q_i$ is defined in \eqref{defQj}. For a proof see  Proposition \ref{sumerrorrot}.
	Furthermore, arguing as in the proof of Proposition \ref{even-ness}, we can prove that the function $\tilde g(z)= g (P_1 + Mz)$ is even in $z_2$, that is
	\be \label{tildeg}
	\tilde g(z_1, z_2) = \tilde g (z_1 , -z_2).
	\ee
	We postpone the proof of \eqref{tildeg} to the Appendix \ref{AppeA}.
	
	\medskip
	
	We modify  $\eta_{0} \sum_{j=1}^N \,  \,  \,  \psi_j $ adding a term which is defined globally in the entire space $\R^2$ and which cancels the error term $g(x)$. 
	Let $H_{2\ve} (x)$ solve
	\be \label{H2e}
	L (H_{2\ve}) + g = 0 , \quad {\mbox {in}} \quad \R^2.
	\ee
	For a smooth function $h(x)$ satisfying the decay condition
	$$
	\| h\|_{\bar \nu}\, :=\, \sup_{x\in \R^2} (1+|x|)^{\bar \nu}|h(x)|\, <\,+\infty \, ,
	$$
	for some ${\bar \nu} >2$, there exists a solution $\psi(x)$ to problem 
	$$
	L(\psi) + g = 0 , \quad {\mbox {in}} \quad \R^2
	$$
	which is of class $C^{1,\beta}(\R^2)$ for any $0<\beta<1$,  and defines a linear operator $\psi = {\mathcal T}^o (g) $ of $h$
	and satisfies 
	\[
	|\psi(x)| \,\le \,     C{ \| g\|_{\bar \nu}} (1+ |x|^2),
	\]
	for some positive constant $C$. The proof of this fact can be found in Lemma \ref{prop2}. 
	
	\medskip
	Using this result we obtain that the solution $H_{2\ve}$ to \eqref{H2e} satisfies 
	\begin{equation}
		\label{bH2}
		H_{2\ve} (Q_{2\pi \over N} x ) = H_{2\ve} (x) , \quad | H_{2\ve } (x) | \leq C (1+ |x|^2).
	\end{equation}
	Since $H_{2\ve}$ solves a linear equation, we use \eqref{tildeg} to conclude that the function
	\be \label{simH2ep}
	z \to H_{2\ve} (P_1 + Mz)
	\ee
	is even with respect to the variable $z_2$.
	
	Besides, observe that such  solution is given up to the addition of a constant.
	We define the  function $H_{2\ve} (x)$  to be the one which furthermore satisfies
	\be \label{H2epP}
	H_{2\ve} ( P_1 ) = 0.
	\ee
	With this is mind, we get to the definition of a globally defined approximate stream function for $N$-helical filaments
	\be \label{psi*}
	\psi_* (x) = \eta_0 (x) \sum_{j=1}^N \, 
	\,  \, \psi_j (x) +H_{2\ve}(x),
	\ee
	so that for $z=M^{-1}(x-P_1)$,
	\begin{equation} \label{psi01}
		\begin{aligned}
			L\left( \psi_*   \right)& = 
			\eta_0   \, \,   \left( \Delta \Gamma_{\ve \mu}  + {4 R_\ve (3h^2+ R_\ve^2) \over h(h^2+ R_\ve^2)^{3\over 2} }  { \ve^2 \mu^2 \, z_1 \over (\ve^2 \mu^2 + |z|^2 )^2}  \right) \\
			&+\eta_0\sum_{j\neq i}^N \,\, \,   \left( \frac{8\ve^2 \mu^2}{(\ve^2 \mu^2+|M_j^{-1}M z-M_j^{-1}(P_j-P_1)|^2 )^2} \right. \\
			&+ \left. {4 R_\ve (3h^2+ R_\ve^2) \over h(h^2+ R_\ve^2)^{3\over 2} }  { \ve^2 \mu^2 \, [M_j^{-1}Mz-M_j^{-1}(P_j-P_1)]_1 \over (\ve^2 \mu^2 + |M_j^{-1}Mz-M_j^{-1}(P_j-P_1)|^2 )^2}  \right) .
		\end{aligned}
	\end{equation}
	We recall that the definition of $\eta_0$ is given in \eqref{defeta} and that $R_\ve = {r \over \sqrt{|\log \ve |}}$.
	
	\medskip
	The function $\psi_*$ is a smooth function defined in the whole $\R^2$ and
	satisfies
	\be \label{simpsi*}
	z \to \psi_* (P_1 + Mz)
	\ee
	is even with respect to $z_2$, as consequence of Proposition \ref{even-ness} and \eqref{simH2ep}.
	We make a final adjustment to $\psi_*$ with a precise choice of the scaling parameter $\mu$ in terms of the points $P_i$, assuming $P_i$ have the form \eqref{points} and satisfy \eqref{defR}.
	
	\subsection{Choice of $\mu$}\label{sub4}
	
	We take $\mu$ given by the relation
	\be\label{defmu}
	2\log \mu= \sum_{j\neq i}^N \Psi_{ \ve \mu}(M_j^{-1}(P_i-P_j))+H_{2\ve}(P_i)  - {\alpha \over 2} \, |\log \ve | \, R_\ve^2 
	\ee
	This definition is well-posed in the sense that it is  independent of the choice of $i$. To see this, we observe that  
	\[
	\sum_{j\neq i}^N \Psi_{ \ve \mu}(M_j^{-1}(P_i-P_j))=\sum_{j\neq k}^N \Psi_{ \ve \mu}(M_j^{-1}(P_k-P_j))
	\]
	for any $i$ and $k$, see Appendix \ref{AppeA} (Proposition \ref{proppsi}). Besides, $H_{2\ve}$ 
	is invariant under rotation of angle ${2\pi \over N}$ (see \eqref{bH2})
	and hence  $H_{2\ve}(P_i)=H_{2\ve}(P_k)$, for all $i$, $k$.

	A direct inspection to \eqref{defmu} using \eqref{alpha} gives
	$$
	\log\mu^2 = 2(N-1)\log(|\log \ve|) +  \overline m (\alpha ), 
	$$
	where $\overline m$ is a smooth function which is uniformly bounded 
	for values of $\alpha$ satisfying \eqref{alpha}, as $\ve \to 0$.
	Hence $\mu$ satisfies the bound in \eqref{defR}.
	The above choice of $\mu$ is to eliminate the terms of size $\log |\log \ve|$ in the expansion of $\psi_*(x) -{\alpha \over 2} |\log \ve| \, |x|^2$ in \eqref{psi0} when computed around a vortex $P_i$. Thanks to the dihedral symmetry, it is enough to check this around the first point $P_1$.

	\begin{prop}\label{ppp}
		Assume the validity of \eqref{points}-\eqref{defmu} and define $r_0= \min_{\ell \not= j} {\mbox {dist}} (r e^{\theta_j}, r e^{i\theta_\ell})$. Let $\delta >0$ such that $\delta <{r_0 \over 4}$.
		For 
		$$
		y= M^{-1} \left({x-P_1 \over \ve \mu}\right), \quad |y|<{\delta \over \ve \mu \sqrt{|\log \ve |}}
		$$
		it holds
		\be \label{psi*exp}
		\begin{aligned}
			\psi_* (x )&= \Gamma(y)(1+c_{1} \ve \mu y_1+c_{2} \ve^2 \mu^2|y|^2)  - 4 \log \ve - 2 \log \mu  + {\alpha \over 2} \, |\log \ve | \, R_\ve^2 \\
			&+ \ve \mu y_1 \Biggl[ - 4 c_1 \log \ve \mu +\pp_1 H_{2\ve}(P_1) \\
			& \quad \quad -{2 (N-1) \over r} \sqrt{|\log \ve |} (1+ {\log |\log \ve | \over |\log \ve |} \, {\mathcal B} (\alpha)  ) \Biggl]  \\
			& 
			+O\, |\log \ve | \, \ve^2 \, \mu^2 \, |y|^2 \, {\mathcal B} (\alpha),
		\end{aligned}
		\ee
		where $\Gamma$ is defined in \eqref{defGamma0}, ${\mathcal B}$ denotes a generic smooth function, which is uniformly bounded with its derivative, for $\alpha$ in the range defined by \eqref{alpha}.  We recall that $\psi_*$ is defined in \eqref{psi*},  $c_1$, $c_2$ are defined in \eqref{defpsi2},  $P_1 = R_\ve (1,0)$, $R_\ve = {r \over \sqrt{|\log \ve |}}.$  
	\end{prop}
	
	\begin{proof}
		For $z=M^{-1} (x-P_1)$ we consider the small region 
		$
		|z| < {\delta \over \sqrt{|\log \ve|}}.
		$
		Using \eqref{psi*} and \eqref{defpsi2}, 
		with $y= {z \over \ve \mu}$ in the region  $|y|<{\delta \over \ve \mu \sqrt{|\log \ve |}}$, we obtain 
		\be \label{expan11}
		\begin{aligned}
			&\psi_* (x )= \Gamma (y)(1+c_{1} \ve \mu y_1+c_{2} \ve^2 \mu^2|y|^2)  - 4 \log \ve \mu \\
			&- 4 \log \ve \mu \left( c_{1} y_1 \, \ve \mu + c_{2} \ve^2 \mu^2 |y|^2 \right) +  \frac{4R_\ve ^3}{h (h^2+R_\ve^2)^{3/2}}H_{1\ve}  (\ve \mu y)\\
			& + H_{2\ve} (\ve \mu M y+P_1)
			+\sum_{j\geq 2}^N \Psi_{ \ve \mu}(M_j^{-1}(P_1-P_j+M\ve\mu y)).
		\end{aligned}
		\ee
		We recall that  $M$, $M_j$ are defined in \eqref{matrixdef} and \eqref{defMj}.
		
		In the region $|y|<{\delta \over \ve \mu \sqrt{|\log \ve |}}$  and under the assumptions \eqref{points}-\eqref{defR}, we have the validity of the following expansions
		\[
		\begin{aligned}
			H_{2\ve}(P_1+\ve\mu y) &=H_{2\ve}(P_1) +\ve\mu y\cdot\nabla H_{2\ve}(P_1) +O(\ve^2\mu^2|y|^2)\\
			H_{1\ve }(\ve\mu y)&=O(\ve^3\mu^3|y|^3), \quad {\mbox {as}} \quad \ve \to 0.
		\end{aligned}
		\]
		Replacing this information and \eqref{defmu} in \eqref{expan11} we get
		\be \label{expan111}
		\begin{aligned}
			&\psi_* (x )= \Gamma (y)(1+c_{1} \ve \mu y_1+c_{2} \ve^2 \mu^2|y|^2)  - 4 \log \ve - 2 \log \mu  + {\alpha \over 2} \, |\log \ve | \, R_\ve^2 \\
			&- 4 \log \ve \mu \left( c_{1} y_1 \, \ve \mu + c_{2} \ve^2 \mu^2 |y|^2 \right)  + \ve\mu y\cdot\nabla H_{2\ve}(P_1) +O(\ve^2\mu^2|y|^2)\\
			&   +\sum_{j\geq 2}^N \left(\Psi_{ \ve \mu}(M_j^{-1}(P_1-P_j+M\ve\mu y)) - \Psi_{ \ve \mu}(M_j^{-1}(P_1-P_j)) \right),
		\end{aligned}
		\ee
		where we have used that $H_{2\ve} (P_1)=0$, see \eqref{H2epP}.
		
		Next we analyze the last term in \eqref{expan111}. We first Taylor expand
		\[
		\begin{aligned}
			&\Psi_{ \ve \mu}(M_j^{-1}(P_1-P_j+M_i\ve\mu y))- \Psi_{ \ve \mu}(M_j^{-1}(P_1-P_j))\\
			&= \ve \mu \nabla \Psi_{ \ve \mu}(M_j^{-1}(P_1-P_j))\cdot M_j^{-1} M y +O(|\log \ve |^{-1} \, \ve^2 \mu^2 |y|^2) ,
		\end{aligned}
		\]
		for $|y|<{\delta \over \ve\mu \sqrt{|\log \ve |}}.$
		A direct computation gives 
		\be\label{expan1b}
		\begin{aligned}
			\nabla \Psi_{ \ve \mu}&(M_j^{-1}(P_1-P_j))\cdot M_j^{-1}M y= \hat a_j+\hat b_j +\hat c_j \quad {\mbox {where}} \\
			\hat a_j &=-\frac{4 M_j^{-1}(P_1-P_j)\cdot M_j^{-1}M y}{(\ve^2\mu^2+|(M_j^{-1}(P_1-P_j)|^2)} \times \\ & \times \Biggl(1
			+c_1(M_j^{-1}(P_1-P_j))_1+c_2|M_j^{-1}(P_1-P_j)|^2\Biggr) \\
			\hat b_j &=\log\frac{8}{(\ve^2\mu^2+|M_j^{-1}(P_1-P_j)|^2)^2} \times \\
			& \times \Biggl(c_1[M_j^{-1}M y]_1
			+2c_2M_j^{-1}(P_1-P_j) \, \cdot \, M_j^{-1} M y\Biggr) \\
			\hat c_j & =\frac{4R_\ve ^3}{h (h^2+R_\ve^2)^{3/2}}\nabla H_{1\ve } (M_j^{-1}(P_1-P_j))\cdot M_j^{-1} M y.
		\end{aligned}
		\ee
		Substituting these expressions in \eqref{expan111}, we obtain
		\begin{align*}
			&\psi_* (x )= \Gamma (y)(1+c_{1} \ve \mu y_1+c_{2} \ve^2 \mu^2|y|^2)  - 4 \log \ve - 2 \log \mu  + {\alpha \over 2} \, |\log \ve | \, R_\ve^2 \\
			&- 4 \log \ve \mu \left( c_{1} y_1 \, \ve \mu + c_{2} \ve^2 \mu^2 |y|^2 \right)  + \ve\mu y_1 \pp_1 H_{2\ve}(P_1) 
			+\ve \mu  \sum_{j\neq 1}^N \left( \hat a_j + \hat b_j + \hat c_j  \right) +O(\ve^2 \mu^2 |y|^2) ,
		\end{align*}
		where we have used \eqref{bH2}.
		We next analyze the term $\sum_{j\neq 1}^N \left( \hat a_j + \hat b_j + \hat c_j  \right)$, see \eqref{expan1b}. Recall that $M_j = Q_j \, M$, with $Q_j$ as in \eqref{defQj}.
		Since
		\begin{align*}
			M&= I+ \left(\left( 1+ {R^2_\ve \over h^2}\right)^{-{1\over 2}}  -1 \right) \, \hat B, \quad  
			M^{-1}= I+ \left( \left( 1+ {R^2_\ve \over h^2}\right)^{{1\over 2}}  -1 \right) \, \hat B
		\end{align*}
		with $\hat B=\left( \begin{matrix} 1 & 0\\
			0  & 0\end{matrix}\right),$ we get $M_j= Q_j + \left(\left( 1+ {R^2_\ve \over h^2}\right)^{-{1\over 2}}  -1 \right)  \,  \left( \begin{matrix} \cos \theta_j & 0\\
			\sin \theta_j  & 0\end{matrix}\right),$
		\begin{align*}
			M_j^{-1} &= Q_j^{-1} + \left( \left( 1+ {R^2_\ve \over h^2}\right)^{{1\over 2}}  -1 \right) \, \left( \begin{matrix} \cos \theta_j & -\sin \theta_j\\
				0  & 0\end{matrix}\right) \\
			{\mbox {and for any }} i,j &\\ 
			M_j^{-1} M_i= M^{-1} &Q_j^{-1} Q_i M = Q_j^{-1} Q_i + \left( \left( 1+ {R^2_\ve \over h^2}\right)^{{1\over 2}}  -1 \right) \, \left( \begin{matrix} 0 &  0 \\
				\sin (\theta_i - \theta_j)  & 0\end{matrix}\right) \\
			&
			+\left(\left( 1+ {R^2_\ve \over h^2}\right)^{-{1\over 2}}  -1 \right)  \,  \left( \begin{matrix} 0 &  - \sin (\theta_i - \theta_j) \\
				0  & 0\end{matrix}\right) .
		\end{align*}
		These facts give
		\be \label{facts}
		\begin{aligned}
			M_j^{-1} (P_1 - P_j) &= R_\ve \left( \begin{matrix}  \left( 1+ {R^2_\ve \over h^2}\right)^{{1\over 2}}  \, ( \cos \theta_j - 1 ) \\
				-\sin \theta_j \end{matrix}\right) \\
			M_j^{-1} M y &= \left( \begin{matrix} \cos \theta_j y_1 + \left( 1+ {R^2_\ve \over h^2}\right)^{-{1\over 2}} \sin \theta_j y_2 \\
				-  \left( 1+ {R^2_\ve \over h^2}\right)^{{1\over 2}} \, \sin \theta_j y_1 + \cos \theta_j y_2\end{matrix}\right)\\
			M_j^{-1} (P_1 - P_j) \cdot   M_j^{-1} M y &= R_\ve \left[ \left( 1+ {R^2_\ve \over h^2}\right)^{{1\over 2}} (1- \cos \theta_j )y _1 - \sin \theta_j y_2\right].
		\end{aligned}
		\ee
		We can then find the expressions for $\hat a_j$, $\hat b_j$ and $\hat c_j$ in \eqref{expan1b}. Using \eqref{facts} we get
		\[
		\begin{aligned}
			\hat a_j &=y_1 \hat a_{j1} + y_2 \hat a_{j2} \\
			\hat a_{j1}&= -\frac{2}{R_\ve} \, \frac{( 1+ {R^2_\ve \over h^2})^{{1\over 2}} (1-\cos \theta_j ) \,  }{ (1-\cos \theta_j ) + {R_\ve^2 h^2 \over 2} (1-\cos \theta_j)^2 + {\ve^2 \mu^2 \over 2} } \, \tilde a_j   \\
			\hat a_{j2}&= \frac{2}{R_\ve} \, \frac{ \sin \theta_j \, }{ (1-\cos \theta_j ) + {R_\ve^2 h^2 \over 2} (1-\cos \theta_j)^2 + {\ve^2 \mu^2 \over 2} } \, \tilde a_j , \quad {\mbox {where }}\\
			\tilde a_j= 1
			+c_1 \, R_\ve \,  &\left( 1+ {R^2_\ve \over h^2}\right)^{{1\over 2}}  \, ( \cos \theta_j - 1 ) 
			+c_2\, R_\ve^2 \left(2 (1-\cos \theta_j) + {R_\ve^2 \over h^2} (1-\cos \theta_j)^2 \right).
		\end{aligned}
		\]
		For $\hat b_j$, we get
		\[
		\begin{aligned}
			\hat b_j &=  y_1 \hat b_{j1} + y_2 \hat b_{j2} \\
			\hat b_{j1}&= \left( c_1 \cos \theta_j + 2 c_2 R_\ve ( 1+ {R^2_\ve \over h^2})^{{1\over 2}} (1- \cos \theta_j ) \right) \, \tilde b_j\\
			\hat b_{j2} & = \left( c_1  ( 1+ {R^2_\ve \over h^2})^{-{1\over 2}} \sin \theta_j -2c_2 R_\ve \sin \theta_j  \right) \, \tilde b_j\\
			\tilde b_j&= \log\frac{8}{(R_\ve^2 [ 2(1-\cos \theta_j) +{R_\ve^2 \over h^2} (1-\cos \theta_j^2) ]+\ve^2\mu^2)^2}.
		\end{aligned}
		\]
		We use the form of the function $H_{1\ve}$ to treat  $\hat c_j$. We recall from Proposition \ref{prop02} that $H_{1\ve}$ has the form $H_{1\ve } (y) = h_1 (|y|) \, \cos 3 \theta$, where $y = |y| e^{i\theta}.$ Hence
		$$
		\nabla H_{1\ve } (y) = \left( 
		\begin{matrix}  h_1' {y_1 \over |y|} \cos 3 \theta + 3h {y_2 \over |y|^2}  \sin 3 \theta\\
			h_1' {y_2 \over |y|} \cos 3 \theta - 3h {y_1 \over |y|^2}  \sin 3 \theta
		\end{matrix}\right).
		$$
		Writing
		$$
		\nabla H_{1\ve } (M_j^{-1}(P_1-P_j))= \left( 
		\begin{matrix}  \Theta_j^1 (\theta_j) \\ \Theta_j^2 (\theta_j)
		\end{matrix}\right),
		$$
		a direct inspection gives that
		\be \label{Thetaj}
		\Theta_j^1 (-\theta_j) = \Theta_j^1 (\theta_j) , \quad 
		\Theta_j^2 (-\theta_j) =- \Theta_j^2 (\theta_j).
		\ee
		We then get that
		\[
		\begin{aligned}
			\hat c_j & = y_1 \hat c_{j1} + y_2 \hat c_{j2} \\
			\hat c_{j1}&= \frac{4R_\ve ^3}{h (h^2+R_\ve^2)^{3/2}}\left( \Theta_j^1 (\theta_j) \cos \theta_j -  \Theta_j^2 (\theta_j) ( 1+ {R^2_\ve \over h^2})^{{1\over 2}} \, \sin \theta_j  \right)\\
			\hat c_{j2}& = \frac{4R_\ve ^3}{h (h^2+R_\ve^2)^{3/2}}
			\left( \Theta_j^1 (\theta_j ) ( 1+ {R^2_\ve \over h^2})^{-{1\over 2}} \sin \theta_j  + \Theta_j^2 (\theta_j) \cos \theta_j   \right).
		\end{aligned}
		\]
		We claim that
		\be \label{ajbjcj0}
		\sum_{j=2}^N \hat a_{j2} =\sum_{j=2}^N \hat b_{j2} =\sum_{j=2}^N \hat c_{j2} =0 .
		\ee
		To prove this, we use that  $\hat a_{j2}$ and $\hat b_{j2}$ have the form $\sin \theta_j \, f(\cos \theta_j )$, for some explicit function $f$. 
		Indeed, if $N$ is even, we can write $N=2n$, and then $\sin \theta_{n+1} =\sin \left( {2\pi \over N} n\right) =0$ and if $N$ is odd we set $N=2n+1$, therefore 
		\begin{align*}
			\sum_{j=2}^N \hat a_{j2} &= \left( \sum_{j=2}^{N-n} + \sum_{j=n+2}^{N} \right)   \sin \theta_j \, f(\cos \theta_j ) \\ 
			&=  \sum_{j=2}^{N-n}   \left(  \sin \theta_j \, f(\cos \theta_j )
			+ \sin \theta_{n+j} \, f(\cos \theta_{n+j} ) \right)\\
			&=  \sum_{j=2}^{N-n}    \left(  \sin \theta_j \, f(\cos \theta_j )
			+ \sin (2\pi - \theta_{(N-n) -j+2} ) \, f(\cos (2\pi - \theta_{(N-n) -j+2} )  ) \right)\\
			&= \sum_{j=2}^{N-n}    \left(  \sin \theta_j \, f(\cos \theta_j )
			+ \sin (2\pi - \theta_{j} ) \, f(\cos (2\pi - \theta_{j} )  ) \right)\\
			&= \sum_{j=2}^{N-n}    \left(  \sin \theta_j \,
			+ \sin (2\pi - \theta_{j} ) \,  \right) \,  f(\cos \theta_j )=0.
		\end{align*}
		A similar argument shows the validity of \eqref{ajbjcj0} for $\sum_{j=1}^N \hat b_{j2}=0.$ In order to prove \eqref{ajbjcj0} for $c_{j2}$, we use the properties in \eqref{Thetaj} for the functions $\Theta_j^1$ and $\Theta_j^2$ and we argue as before.
		From \eqref{ajbjcj0}, we conclude that
		\[
		\sum_{j=1}^N \hat a_j = y_1 \, \sum_{j=2}^N \hat a_{j1}, \quad \sum_{j=1}^N \hat b_j = y_1 \, \sum_{j=2}^N \hat b_{j1}, \quad \sum_{j=1}^N \hat c_j = y_1 \, \sum_{j=2}^N \hat c_{j1}.
		\]
		Under the assumptions on the points $P_j$ in \eqref{points}-\eqref{defR}, recalling that $R_\ve ={r \over \sqrt{|\log \ve |}} $ we recognize that
		\begin{align*}
			\sum_{j=2}^N \hat a_{j1}&= - {2 (N-1) \over r} \sqrt{|\log \ve |} \left( 1+ O( {1\over |\log \ve |} ) \right) \\
			\sum_{j=2}^N \hat b_{j1}&= O({\log |\log \ve| \over \sqrt{|\log \ve |}} )\\
			\sum_{j=2}^N \hat c_{j1}&= O({1\over |\log \ve|^{3 \over 2}}) 
		\end{align*}
		where ${\mathcal B}$ denotes a generic smooth function, which is uniformly bounded with its derivative, for $r \in (r_0 , r_0^{-1})$, as $\ve \to 0$. 
	\end{proof}
	
	\begin{remark}\label{rmk1}
		Under the assumptions of Proposition \ref{ppp}, we obtain, for 
		$$
		y= M^{-1} \left({x-P_1 \over \ve \mu}\right), \quad |y|<{\delta \over \ve \mu \sqrt{|\log \ve |}}, \quad 0<\delta <{r_0 \over 4}
		$$
		that
		\be \label{psi*exp-alpha}
		\begin{aligned}
			\psi_* (x )-{\alpha \over 2} \, |\log \ve | \, |x|^2 &= \Gamma(y)(1+c_{1} \ve \mu y_1+c_{2} \ve^2 \mu^2|y|^2)  - 4 \log \ve - 2 \log \mu \\
			&+ \ve \mu y_1 \Biggl[ - 4 c_1 \log \ve \mu +\pp_1 H_{2\ve}(P_1) -\alpha {h R_\ve \over \sqrt{h^2 + R_\ve^2}} \, |\log \ve | \\
			& \quad \quad \quad -{2 (N-1) \over r} \, \sqrt{|\log \ve |}  \, (1+ {\log |\log \ve | \over |\log \ve |} {\mathcal B} (\alpha ) ) \Biggl]  \\
			& 
			+\, |\log \ve | \, \ve^2 \, \mu^2 \, |y|^2 \, {\mathcal B} (\alpha )  ,
		\end{aligned}
		\ee
		where $\Gamma$ is defined in \eqref{defGamma0}, ${\mathcal B}$ denotes a generic smooth function, which is uniformly bounded with its derivative, for $\alpha$ in the range defined by the relation \eqref{alpha}.
		
		Expansion \eqref{psi*exp-alpha} uses 
		$$
		\begin{aligned}
			|x|^2&= R_\ve^2 + 2 \, \ve \,\mu \, P_1 \cdot My + |My|^2 \\
			&=R_\ve^2 + 2 \, \ve \,\mu \, {  h \, R_\ve \over  \sqrt{h^2+  R_\ve^2} }  \, y_1 + \, \ve^2 \,\mu^2 \, \left( {  h^2 \over  h^2+  R_\ve^2}  y_1^2 + \,  y_2^2 \right).
		\end{aligned}
		$$

	\end{remark}

	\section{Choice of the non-linearity and construction of the first approximate solution} \label{sec4}

	In this section we define a  nonlinearity $F$ in problem \eqref{P}
	with the property that the vorticity $W$, defined as
	$$
	W(x) = F (\Psi -{\alpha \over 2} \, |\log \ve | \, |x|^2)$$
	satisfies 
	$$
	W(x) \sim 8 \pi \sum_{j=1}^N  \delta_{P_j}, \quad {\mbox {as}} \quad \ve \to 0.
	$$
	We recall problem \eqref{P} here
	$$
	\nabla\cdot (K \nabla \Psi)+F(\Psi -{\alpha \over 2} \, |\log \ve | \, |x|^2)=0
	$$
	Take $\delta >0$ to be a fixed positive number and consider the inner region around $P_i$ to be given by 
	\be \label{innreg1}
	|M_i^{-1}(x-P_i)| <{\delta \over {\sqrt{|\log \ve|}}.
	}
	\ee
	We take $\delta< \min \{{r_0 \over 4} , {1\over 2} \}$, where $r_0$ is defined in the statement of Proposition \ref{ppp}, so that for $x \in \R^2 $ satisfying  \eqref{innreg1} then $\eta_0 (x) =1$. In addition, we consider
	\be \label{varinner}
	x-P_i=M_iz,\quad  z=\ve\mu y.
	\ee

	\subsection{Choice of the non-linearity $F$}
	
	We let $F$
	\be \label{defF}
	F(s ) =  \ve^{2}  \eta ( s)\, e^s,
	\ee
	where
	$\eta$ is a smooth cut-off functions defined as follows.

	\medskip
	
	Consider the boundary of the inner region around $P_i$, as defined in \eqref{innreg1}. Using the variable $y$ in \eqref{varinner}, this boundary is defined by $|y|=\delta/({\mu\ve\sqrt{|\log\ve|}})$. On this boundary we have
	\begin{align*}
		&
		\psi_* (x ) -{\alpha \over 2} \, |\log \ve | |x|^2
		= 2 \log|\log \ve|+2\log\mu+\log 8- 4\log\delta  \,  (1+ s (\delta ) ) +o (\ve )
	\end{align*}
	where $s(\delta)$ is a smooth bounded function of $\delta $ and $o(\ve ) \to 0 $ is respect to $\ve\to 0$. This follows from  formula \eqref{psi*exp}.
	We choose the cutoff function $\eta$ such that 
	\begin{equation}\label{cutoff}
		\begin{aligned}
			\eta(s)&=1, \quad {\mbox {for}} \quad  s\geq 2 \log|\log \ve|+2\log\mu +\log8+2d_{\ve} \\
			\eta(s)&=0 , \quad {\mbox {for}} \quad s\leq 2 \log|\log \ve|+2\log\mu +\log8 +d_{\ve}
		\end{aligned}
	\end{equation}
	for suitable $d_{\ve}=-4\log\delta  \,  (1+ s (\delta ) ) +o (\ve )$ as $\ve \to 0$ so that 
	$$
	\eta\left( \psi_* (x ) -{\alpha \over 2} \, |\log \ve | |x|^2\right)=1\quad \mbox{if}\quad |M_i^{-1}(x-P_i)|\leq \frac{\delta^2}{\sqrt{|\log \ve|}}
	$$
	for all $i:=1,\ldots,N$
	and 
	\be \label{cutout}
	\eta\left( \psi_* (x ) -{\alpha \over 2} \, |\log \ve | |x|^2\right)=0\quad \mbox{if}\quad  x \in \cap_{i=1}^N \left\{ |M_i^{-1}(x-P_i)|\geq \frac{\delta}{\sqrt{|\log \ve|}} \right\}.
	\ee
	Here $\delta$ is independent of $\ve$ as in \eqref{innreg1} and can be taken smaller if needed. Notice that the support of the function $F (\, \psi_* -{\alpha \over 2} \, |\log \ve | |x|^2 \, )$ is contained well inside the unitary ball of radius $1$ in $\R^2$. This fact will be crucial in the sequel.
	
	\subsection{Estimate of the error function}
	Let us define  the error function to be
	\be \label{error}
	S[\psi ] (x) = L (\psi) + \ve^{2}\eta \left(\psi -{\alpha \over 2} \, |\log \ve | |x|^2 \right)\, e^{\psi -{\alpha \over 2} \, |\log \ve | |x|^2}
	\ee
	A solution to \eqref{P} would correspond to a smooth function $\psi$ such that
	\[
	S[\psi ] (x) = 0 \quad x\in \R^2.
	\]
	Our purpose is to estimate 
	\[
	S[\psi_* ] (x) \quad {\mbox {for}} \quad  x \in \R^2
	\]
	where $\psi_*$ is the approximate stream function for the $N$-helical filaments introduced in \eqref{psi*}.

	\medskip
	\begin{prop} \label{prop4}
		Assume the validity of \eqref{points}-\eqref{alpha}-\eqref{defmu}. Let  $0<\delta< \min \{{r_0 \over 4} , {1\over 2} \}$, where $r_0$ is defined in the statement of Proposition \ref{ppp},  and consider the cut-off function $\eta$ as defined in \eqref{cutoff}. There exist $\bar \nu >2$, $\sigma \in (0,1)$, $a \in ({3\over 4} ,1)$ and  $C>0$ such that the following estimates hold true.
		
		In the region $\{ x \in \R^2 : |M_i^{-1}(x-P_i)|>{\delta \over \sqrt{|\log \ve|}}, \quad \forall i =1, \ldots , N\} $  we have
		\be
		\label{SPsi02}
		|S(\psi_*) (x)|\leq C\frac{\ve^{1+\sigma}}{1+|x|^{\bar \nu }} 
		\ee
		as $\ve \to 0$.
		For $i=1, \ldots , N$, in each region defined by 
		\[
		|M_i^{-1}(x-P_i)|<{\delta \over \sqrt{|\log \ve|}}, 
		\]
		we have
		\be
		\label{SPsi01a}
		\ve^2\mu^2 |S(\psi_*) (x) |\leq C  {  \ve \mu \sqrt{|\log \ve |} \over  1 + |y|^{2+a}} 
		\ee
		as $\ve \to 0$. 
		In \eqref{SPsi01a}, $y$ is the scaled variable in the inner regions:
		$$x-P_i = M_i z, \quad z = \ve \mu y.$$
		Besides, in the region given by $|M^{-1}(x-P_1)|<{\delta^2 \over \sqrt{|\log \ve|}}$  we have
		\be \label{ee2n}
		\begin{aligned}
			S(\psi_* ) (x) &=\frac{1}{\ve^2\mu^2}U (y)  \Biggl[ \ve \mu y_1 \left(c_{1} \Gamma (y) + {\mathcal A} (\alpha) \right) \\
			&  + \ve^2 \mu^2 c_{2} |y|^2 \Gamma (y) +O\left( \log(|\log \ve |)|y|^2 \ve^2\mu^2
			\right) \Biggl] 
			+O(\ve^2\mu^2|\log\ve|^2),
		\end{aligned}
		\ee
		for $y=M^{-1}(x-P_1)/\ve\mu$, where
		\be \label{defA}
		{\mathcal A} (\alpha) =2\sqrt{|\log \ve |} \left( {  r \over  h^2}  - { (N-1) \over r} - {\alpha r \over 2} \right)+
		\frac{\log|\log\ve|}{\sqrt{|\log\ve|}}
		\, {\bf Y}_1 (\alpha).
		\ee
		Here ${\bf Y}_1(\alpha)$ denotes an explicit smooth function,  uniformly bounded  as $\ve \to 0$  for values of $\alpha$     satisfying  \eqref{alpha}. 
	\end{prop}

	\medskip
	\begin{proof}
		Let $\delta >0$ be fixed small, and consider first the inner region around $P_1$ defined as 
		$$
		\left\{ x \in \R^2 \, : \, |M^{-1}(x-P_1)|\leq \frac{\delta^2}{\sqrt{|\log \ve|}} \right\}.
		$$
		By dihedral symmetry, it is enough to show the validity of estimate \eqref{SPsi01a} in this region to get the equivalent estimates in inner regions around any other point $P_i$, 
		for $i \in \{ 2, \ldots , N \}$.
		We split this region into two parts
		$$ |M^{-1}(x-P_1)|\leq \frac{\delta^2}{\sqrt{|\log \ve|}} ,  \quad \frac{\delta^2}{\sqrt{|\log \ve|}}\leq |M^{-1}(x-P_1)|\leq \frac{\delta}{\sqrt{|\log \ve|}}.$$
		We recall the definition of $M$ in \eqref{matrixdef}. It corresponds to $M_1$ in \eqref{defMj}.
		
		\medskip
		Assume first that $|M^{-1}(x-P_1)|\leq \frac{\delta^2}{\sqrt{|\log \ve|}} $. 
		Thanks to the result of Proposition \ref{ppp} and  the expansion in \eqref{psi*exp}, we get that the non-linear term in the expression of $S(\psi_*)$ becomes
		\be \label{rhsf}
		\begin{aligned}
			& \ve^{2}  \eta \left(\psi_*(x) -{\alpha \over 2} |\log \ve | |x|^2 \right) e^{\psi_* (x) -{\alpha \over 2} |\log \ve | |x|^2}
			= \frac{1}{\ve^2\mu^2}U (y)  \Biggl[ 1+ \ve \mu y_1 \left(c_{1} \Gamma (y) + {\mathcal A}_{1} (\alpha) \right)   \\
			& \qquad  \qquad \quad  \quad \quad \: + \ve^2 \mu^2 c_{2} |y|^2 \Gamma (y) +O\left( |\log \ve | \, |y|^2 \ve^2\mu^2
			\right) \Biggl]
			&   \\
			&{\mbox {with}}\\
			&{\mathcal A}_{1} (\alpha): =4 c_1 |\log \ve | \,  - 4c_{1} \log\mu - 
			{ 2 (N-1) \over R_\ve }- \alpha {h R_\ve \over \sqrt{h^2 + R_\ve^2}} \, |\log \ve | \\
			& \quad \quad \quad + \frac{\log |\log \ve |}{\sqrt{|\log\ve|}} {\bf Y}_1(\alpha) 
		\end{aligned}
		\ee
		where ${\bf Y}_1(\alpha)$  is a smooth function,  uniformly bounded  as $\ve \to 0$  for values of $\alpha$ satisfying \eqref{alpha}. 
		
		If we insert the definition of   $c_{1} $ as given in \eqref{defpsi2} and recall that $R_\ve = {r \over \sqrt{|\log \ve |}}$, we see that $\mathcal A_{1} (\alpha)$  
		\be\label{aa1}
		\begin{aligned}
			{\mathcal A}_{1} (\alpha) 
			&= \sqrt{|\log \ve |} \left( { 2 h r \over  \sqrt{(h^2+ {r^2 \over |\log \ve |} )^3} } - {2 (N-1) \over r} - \alpha {h r \over \sqrt{h^2 + {r^2 \over |\log \ve |}}}  \right) +
			\frac{\log|\log\ve|}{\sqrt{|\log\ve|}}
			\, {\bf Y}_1 (\alpha)\\
			&= 2\sqrt{|\log \ve |} \left( {  r \over  h^2}  - { (N-1) \over r} - {\alpha r \over 2} \right)+
			\frac{\log|\log\ve|}{\sqrt{|\log\ve|}}
			\, {\bf Y}_1 (\alpha)
		\end{aligned}
		\ee
		where again ${\bf Y}_1(\alpha)$ denotes an explicit smooth function,  uniformly bounded  as $\ve \to 0$  for values of $\alpha$ satisfying \eqref{alpha}. To get \eqref{aa1} we also use the definition and expansion of $\mu$ as given in \eqref{defmu}.

		\medskip
		A direct computation shows that in the region  $\frac{\delta^2}{\sqrt{|\log\ve|}}\leq |M^{-1}(x-P_1)|\leq \frac{\delta}{\sqrt{|\log\ve|}}$ we have
		\begin{align*}
			\ve^{2} \eta \left( \psi_* - {\alpha \over 2} |\log \ve | |x|^2 \right) e^{\psi_* - {\alpha \over 2} |\log \ve | |x|^2}
			=O(\ve^2\mu^2|\log\ve|^2 ). 
		\end{align*}
		On the other hand we recall from \eqref{psi01} we obtain that
		\be \label{ee1}
		\begin{aligned}
			L(\psi_*)&
			=-\frac {1}{\ve^2\mu^2}\left[U(y)-{4 R_\ve (3h^2+ R_\ve^2) \over h(h^2+ R_\ve^2)^{3\over 2} }  { \ve \mu \, y_1 \over (1 + |y|^2 )^2}\right]+O(\ve^2\mu^2|\log\ve|^2)
		\end{aligned}
		\ee
		in the region  $|M^{-1} (x-P_1) |<\delta/\sqrt{|\log \ve|}$, 
		where $y$  is the variable $y=M^{-1}(x-P_1)/\ve\mu$.

		\medskip
		Combining \eqref{rhsf} and \eqref{ee1}, we conclude that in the region
		$|M^{-1} (x-P_1) |<\delta^2/{\sqrt{|\log \ve|}}$ we have
		\begin{align*}
			S(\psi_* ) (x) &=\frac{1}{\ve^2\mu^2}U (y)  \Biggl[ \ve \mu y_1 \left(c_{1} \Gamma (y)+{ R_\ve (3h^2+ R_\ve^2) \over 2h(h^2+ R_\ve^2)^{3\over 2} }  + {\mathcal A}_{1} (\alpha) \right) \\
			&  + \ve^2 \mu^2 c_{2} |y|^2 \Gamma (y) +O\left( \log(|\log \ve |)|y|^2 \ve^2\mu^2
			\right) \Biggl] 
			+O(\ve^2\mu^2|\log\ve|^2)
		\end{align*}
		with $\mu $ given by \eqref{defmu}, and $S$ by \eqref{error}. This gives \eqref{ee2n}-\eqref{defA} and \eqref{SPsi01a}.
		
		\medskip
		Besides
		we get that in the region  $\frac{\delta^2}{\sqrt{|\log\ve|}}\leq |M^{-1}(x-P_1)|\leq \frac{\delta}{\sqrt{|\log\ve|}}$ we have
		$$
		S(\psi_* ) (x)
		= O(\ve^2\mu^2|\log\ve|^2).
		$$
		
		\medskip
		Let us consider now the region defined by
		$$|M_i^{-1}(x-P_i)|>{\delta \over \sqrt{|\log \ve|}}, \quad \forall i.
		$$
		In this outer region
		$
		S[\psi_* ] (x) = L (\psi_*).$
		A direct inspection of \eqref{psi01} gives
		$$
		\left|L\left(\psi_* \right)\right|\leq C\frac{\ve^2\mu^2}{1+|x|^{\bar \nu}}
		$$
		for some $C$ independent of $\ve$, $\bar \nu>2$ and $\mu$ is defined in \eqref{defmu}. Combining all previous results, we conclude the proof of the proposition.
	\end{proof}

	\begin{remark}\label{rmk}
		Let
		$
		x= P_1 + \ve \mu M y $ where $M$ is defined in \eqref{matrixdef}, and define
		$$
		\tilde S_* (y) = 
		S(\psi_*) (P_1 + \ve \mu My).
		$$
		Then
		$$
		\tilde S_* (y_1 , -y_2) = \tilde S_* (y_1 , y_2).
		$$
		This is consequence of \eqref{simpsi*}, the properties of the operator $L$ and the definitions of the non-linearity $F$.
	\end{remark}

	\medskip
	Estimates \eqref{SPsi02} and \eqref{SPsi01a} will be crucial to continue the inner-outer gluing procedure that leads to an exact solution of \eqref{P}. This is what we discuss next.

	\section{The inner-outer gluing system}\label{quattro}
	
	To prove Theorem \ref{teo2} we will show that problem \eqref{P}
	admits a solution of the form
	$$
	\psi = \psi_* + \psi_\#
	$$
	where $\psi_\#$ is a small function compared to $\psi_*$, whose definition can be found in \eqref{psi*}. Equation \eqref{P} can be rewritten in terms of $\psi_\#$ as in
	\be \label{fin}
	L (\psi_\# ) + F' (\psi_* - {\alpha \over 2} \, |\log \ve |\,  |x|^2 ) [\psi_\# ]+ S[\psi_*] + {\mathcal N}(\psi_\#)=0 \quad {\mbox {in}} \quad \R^2
	\ee
	where  $F$ is in \eqref{defF},
	\be \label{nonl}
	\begin{aligned}
		{\mathcal N} (\psi_\# ) &= F \left(\psi_* + \psi_\# - {\alpha \over 2}\,  |\log \ve | \, |x|^2 \right) - F \left(\psi_* - {\alpha \over 2} \, |\log \ve |\,  |x|^2 \right) \\
		&- F' \left(\psi_* - {\alpha \over 2} \, |\log \ve |\,  |x|^2 \right) [\psi_\# ] 
	\end{aligned}
	\ee
	and
	$$
	S[\psi_*] = L \psi_* + F\left(\psi_* - {\alpha \over 2} \, |\log \ve |\,  |x|^2 \right).
	$$
	We recall that $\psi_*$ is invariant under the dihedral symmetry
	\be\label{dihe}
	\psi_* (Q_{2\pi \over N} x ) = \psi_* (x), \quad \forall x \in \R^2,
	\ee
	and by definition of $L$ and $F$, also $S[\psi_*]$ satisfies \eqref{dihe}.
	This section is devoted to describe the strategy we follow to find $\psi_\#$. In next Section we solve \eqref{fin}.
	
	Let  $\eta$ be a smooth cut-off function such that
	$$
	\eta (z) = 1 \quad {\mbox {if}} \quad |z|  \leq 1  , \quad
	\eta (z) = 0 \quad {\mbox {if}} \quad |z| \geq 2   .
	$$
	Let $\delta >0 $ be given as in Proposition \ref{prop4}.
	Fix $\delta_1 >0$ small with $2 \delta_1 <\delta^2$ and define 
	\be \label{eta1}
	\eta_1 (x)  = \eta \left({\sqrt{|\log \ve |} \, M^{-1} (x-P_1) \over \delta_1 }  \right).
	\ee
	Recall that $M$ is the matrix introduced in \eqref{matrixdef}, $P_1$ is the pointy $P_1 = {r \over \sqrt{|\log \ve |}} \, (1,0)$ as in \eqref{points}-\eqref{defR}.

	We write the  function $\psi_\#$ in the form
	\be \label{ansatz}\begin{aligned}
		\psi_\# (x) &= \sum_{j=1}^N \eta_j (x) \Phi_j (x)   + \varphi (x), \\
		\eta_j (x) &= \eta_1 (Q_j^{-1}  x) , \quad \Phi_j (x) = \Phi (Q_j^{-1}  x)\\
		\Phi (x) &= \phi (y) , \quad  y= {M^{-1} (x-P_1) \over \ve \mu } . \quad 
	\end{aligned}
	\ee
	We recall that $Q_j$ is the matrix of rotation defined in \eqref{defQj}.
	
	
	Notice that $\sum_{j=1}^N \eta_1 (Q_j^{-1}  x) \Phi (Q_j^{-1}  x)$ is invariant under the dihedral symmetry \eqref{dihe}. We shall assume that also $\varphi$ satisfies \eqref{dihe}, that is 
	\be\label{dihevar}
	\varphi (Q_{2\pi \over N} x ) = \varphi (x), \quad \forall x \in \R^2.
	\ee
	
	In Proposition \ref{prop-gluing} we introduce a  system of coupled equations in the functions $\phi$ and $\varphi$ which gives a solution to \eqref{fin}. To state the result, let us write 
	\be \label{defbe}
	\ve^2\mu^2F'\left( \psi_* - {\alpha \over 2} \, |\log \ve |\,  |x|^2\right)=e^{\Gamma (y)}+b_\ve (y)\quad \mbox{with}\quad \Gamma (y) = \log {8 \over (1+ |y|^2)^2}.
	\ee
	Since
	$$
	F' \left(\psi_* - {\alpha \over 2} \, |\log \ve |\,  |x|^2 \right) = \ve^2 \left( \eta' (\psi_* - {\alpha \over 2} \, |\log \ve |\,  |x|^2) + \eta (\psi_* - {\alpha \over 2} \, |\log \ve |\,  |x|^2)  \right) \, e^{\psi_*  - {\alpha \over 2} \, |\log \ve |\,  |x|^2},
	$$
	we can bound $|F' (\psi_* - {\alpha \over 2} \, |\log \ve |\,  |x|^2)| \leq F(\psi_* - {\alpha \over 2} \, |\log \ve |\,  |x|^2)$ and 
	using the estimate \eqref{rhsf}, we get
	that
	\be \label{biest}
	b_\ve (y)=O\left(
	\frac{\ve \mu \sqrt{|\log\ve|}}{1+|y|^{2+a}}\right), \quad {\mbox {if}} \quad |y| <{2 \delta_1 \over \ve \mu \sqrt{|\log \ve |}}
	\ee
	for some $a\in (0,1)$ as in Proposition \ref{prop4}. Note also that $b_\ve$ is even in $y_2$, that is 
	\be \label{beven}
	b_\ve(y_1,y_2)=b_\ve(y_1,-y_2)\quad \mbox{for all}\quad y\in \RR^2.
	\ee
	
	\begin{prop}\label{prop-gluing}
		Assume that $\varphi$ satisfies \eqref{dihe} and that $(\phi, \varphi)$ solve the following system of equations
		\be \label{in}
		\Delta_y \phi + e^\Gamma \phi + B_\ve [\phi] + b_\ve (y) \phi + H(\phi , \varphi,\alpha )=0  , \quad {\mbox{for}} \quad |y| <{\delta_1 \over \ve \mu \sqrt{|\log \ve |}}
		\ee
		\be\label{out}
		\begin{aligned}
			L[\varphi ] &  + \left(1-\sum\limits_{i=1}^N\eta_i \right)\left[   F'( \psi_*  - {\alpha \over 2}  \, |\log \ve |\, |x|^2) \varphi  +  S (\psi_* )  +  {\mathcal N} \left(\sum\limits_{i=1}^N\eta_i \Phi_i + \varphi \right) \right]\\
			&+ \sum\limits_{i=1}^N \left( L[\eta_i \Phi_i ] -\eta_i L[\Phi_i]  \right)    =0 \quad  {\mbox{for}} \quad x \in \R^2 .
		\end{aligned}
		\ee  
		Then the function $\psi_\#$ of the form \eqref{ansatz} solves \eqref{fin}.
		In \eqref{in}, $b_\ve(y)$ is the function introduced in \eqref{defbe}, the operator $B_\ve$ is the operator $B$ in \eqref{Lz} expressed in the variable $y$ introduced in \eqref{ansatz} multiplied by $\ve^2\mu^2$,  and 
		\be \label{in-out-1}
		\begin{aligned}
			H(\phi , \varphi , \alpha) &=  \ve^2 \mu^2  {\mathcal N} \left(\sum\limits_{i=1}^N\eta_i \Phi_i + \vp \right) +  \ve^2 \mu^2 F'(\psi_* - {\alpha \over 2} \, |\log \ve |\, |x|^2)\varphi +  \ve^2 \mu^2 S(\psi_*) .
		\end{aligned}
		\ee
	\end{prop}

	\begin{proof}
		We  write $\varphi = \sum_{j=1}^N \eta_j \varphi + (1- \sum_{j=1}^N \eta_j ) \varphi$, where $\eta_j$ are the cut-off functions introduced in \eqref{ansatz}. A straightforward calculation gives
		$$
		\begin{aligned} 
			S( \psi_*+ \psi_\# )&=\
			\sum\limits_{j=1}^N\eta_j\big [  L[\Phi_j]  +  F'( \psi_* - {\alpha \over 2} \, |\log \ve |\, |x|^2) (\Phi_j + \varphi )  + S (\psi_* ) +  {\mathcal N} (\sum\limits_{i=1}^N\eta_i \Phi_i + \varphi ) \big]
			\\
			& +    L[\varphi ]   + \left(1-\sum\limits_{i=1}^N\eta_i \right)\left[   F'( \psi_* - {\alpha \over 2}  \, |\log \ve |\,|x|^2) \varphi  +  S (\psi_* )  +  {\mathcal N} \left(\sum\limits_{i=1}^N\eta_i \Phi_i + \varphi \right) \right]\\
			&+ \sum\limits_{i=1}^N \left( L[\eta_i \Phi_i ] -\eta_i L[\Phi_i]  \right)
		\end{aligned}
		$$
		where ${\mathcal N}$ is defined in \eqref{nonl} and $F$ in \eqref{defF}. Assume  $\Phi_j$, $j=1, \ldots , N$ and $\varphi$ solve the following system of $(N+1)$ equations
		\be \label{ff1}
		\begin{aligned}
			L[\Phi_j] & +  F'( \psi_* - {\alpha \over 2} \, |\log \ve |\, |x|^2 ) (\Phi_j + \varphi )  + S (\psi_* ) +  {\mathcal N} \left(\sum\limits_{i=1}^N\eta_i \Phi_i + \varphi \right)=0 \\
			& {\mbox {for}} \quad |M^{-1} Q_j^{-1}  (x-P_j )| \leq {\delta_1 \over {\sqrt{|\log \ve |} }}, \quad j=1, \ldots , N , \\
			L[\varphi ] &  + \left(1-\sum\limits_{i=1}^N\eta_i \right)\left[   F'( \psi_* - {\alpha \over 2} \, |\log \ve |\, |x|^2) \varphi  +  S (\psi_* )  +  {\mathcal N}\left(\sum\limits_{i=1}^N\eta_i \Phi_i + \varphi \right) \right]\\
			&+ \sum\limits_{i=1}^N \left( L[\eta_i \Phi_i ] -\eta_i L[\Phi_i]  \right)    =0 \quad  {\mbox {for}} \quad x \in \R^2 
		\end{aligned}
		\ee
		Then a direct inspection gives that $\psi_\#$ with the form \eqref{ansatz} is a solution to  \eqref{fin}.
		
		Thanks to the dihedral symmetry of $\psi_*$
		and the form of the functions $\Phi_j $ as in \eqref{ansatz}  given by
		$$
		\Phi_j (x) = \Phi (Q_j^{-1} x ),
		$$
		for a fixed function $\Phi$,
		we solve first the last equation in \eqref{ff1} for functions $\varphi$ satisfying, by hyphotesis \eqref{dihevar}, the dihedral symmetry
		$$
		\varphi (Q_{2\pi \over N} x) = \varphi (x), \quad \forall x \in \RR^2.
		$$
		In fact the first $N$ equations in \eqref{ff1} are reduced to just one equation for $\Phi$.

		Recalling from \eqref{ansatz} that $\Phi (x) = \phi (y)$, where $y= {M^{-1} (x-P_1) \over \ve \mu }$, it is convenient to recast the equation for $\Phi$ in terms of the variable $y$.
		From Proposition \ref{prop1} we have
		$$
		L[\Phi]   =   \frac 1{\ve^2\mu^2} \big[ \,\Delta_y \phi +  B_\ve [\phi]  \,\big ]
		$$
		where $B_\ve[\phi] =\ve^2\mu^2B(\ve\mu y)[\phi]$ and $B$ is the operator given in \eqref{Lz}.
		
		Multiplying  the first equation  in \equ{ff1} by $\ve^2\mu^2$, we get the equation for $\phi$
		$$
		\Delta_y \phi  +   e^{\Gamma} \phi =  -B_\ve[\phi] - b_\ve(y)\phi   -  H(\phi, \vp, \alpha )    \inn B_\rho (0)
		$$
		where $B_\rho (0)$ is the ball centered at $0$ of radius $\rho$, with 
		$\rho= \frac {\delta_1}{\ve\mu\sqrt{|\log\ve|}} $. Here
		$$
		H(\phi, \vp, \alpha)  = \ve^2\mu^2 {\mathcal N}  (\sum\limits_{i=1}^N\eta_i \Phi_i +  \vp)  +\ve^2 \mu^2 \, F'\left(\psi_* - {\alpha \over 2} \, |\log \ve |\, |x|^2 \right) \varphi + \tilde E,
		$$
		the function $b_\ve$ defined in \eqref{defbe} and  $\tilde E= \ve^2 \mu^2 S[\psi_*].$
		This concludes the proof of the proposition.
	\end{proof}

	\medskip
	To solve the coupled system \eqref{in}-\eqref{out}, we proceed as follows: we assume that $\phi$ is given of a special type and solve \eqref{out}. 
	
	The assumptions on $\phi$ are motivated by the following facts on $\tilde E$. 
	From Remark \ref{rmk}, we observe that
	$$
	\tilde E(y) = \ve^2 \mu^2 S[\psi_*] (P_1 + \ve \mu My)
	$$
	satisfies 
	\be \label{rmk1}
	\tilde E(y_1 , y_2)
	= \tilde E(y_1 , -y_2) \quad \forall y\in \RR^2.\ee
	From estimates \eqref{SPsi01a}, 
	we obtain, in the region $|y|<{\delta_1 \over \mu \ve\sqrt{|\log\ve|}}$,
	\begin{equation*}
		\tilde E =
		O\left( \frac{ \ve \mu \sqrt{|\log\ve|}}  {1+|y|^{2+a}}\right) 
	\end{equation*}
	for $a \in (0,1)$ as in Proposition \ref{prop4}.
	With this in mind, we assume
	\[
	\phi(y_1 , y_2) = \phi(y_1, - y_2), \quad \forall y \in \RR^2
	\]
	and that $\phi$ has the size of the error 
	$\tilde E:= \ve^2 \mu^2 S(\psi_*) $
	with two powers less of decay in $y$.
	Thus we solve \eqref{out} assuming that $\phi$ satisfies
	$$
	(1+|y|) |D_y\phi (y)|+  |\phi (y)| \le  \frac {C \ve \mu \sqrt{|\log\ve|} } {1+ |y|^{a}}. 
	$$
	The basic linear operator $\Delta_y \phi + e^{\Gamma} \phi $ in \eqref{in} has  a three-dimensional kernel generated by the bounded functions
	\be \label{defZj}
	Z_i (y) = {\pp \Gamma \over \pp y_i } , \quad i=1,2, \quad Z_0 (y) = 2+ y\cdot \nabla \Gamma (y) ,
	\ee
	see for instance \cite{BaPa}.
	This observation implies that the solvability of \eqref{in} within the anticipated topologies hinges on whether the right-hand side possesses components in the directions spanned by the $Z_i$'s.
	Under the symmetry assumptions \eqref{dihe}, we shall see that no component of the right hand-side in \eqref{in} exists in the direction $Z_2$ (See proof of Proposition \ref{propphi1}).

	\medskip
	Rather than directly addressing \eqref{in}, we will instead tackle the auxiliary projected problem
	\begin{equation}\label{innerp}
		\Delta_y \phi + e^{\Gamma } \phi + B_\ve [\phi ] + b_\ve (y) \phi +H(\phi,\psi,\alpha )=d_1 e^{\Gamma (y)}Z_1\quad \mbox{in}\quad B_\rho,
	\end{equation}
	for some constant $d_1$. We solve \eqref{innerp} coupled with the outer problem \eqref{out}. The constant $d_1$ depends on $\phi$, $\vp$, and $\alpha$.

	To find a solution of \eqref{in}, we will need to show that
	\be \label{eqc1}
	d_1 = d_1 [ \phi , \varphi , \alpha ] =0
	\ee
	We can do this by choosing the value of $\alpha$ in the range of values satisfying \eqref{alpha}. 
	
	\medskip
	The rest of the paper is devoted to find $\vp$, $\phi$, and $\alpha$ that solve the system given by the {\it outer problem} \eqref{out}, the {\it projected inner problem} \eqref{innerp} with \eqref{eqc1}. This is done in \S \ref{sette}. Sections \S \ref{AppeA} and \S \ref{AppeB} contain some technical results.

	\bigskip

	\section{Solving the inner-outer gluing system}\label{sette}
	
	In this section, we solve in $\vp$, $\phi$ and $\alpha$ the system of equations 
	\be\label{inout}
	\begin{aligned}
		\Delta_y \phi &+ e^{\Gamma } \phi + \hat B_\ve [\phi ] + H(\phi,\vp,\alpha)=d_1 e^{\Gamma (y)}Z_1\quad \mbox{in}\quad B_\rho \\
		L[\varphi ] &  + (1-\sum\limits_{i=1}^N\eta_i )\left[   F'( \psi_* - {\alpha \over 2} \, |\log \ve |\, |x|^2 ) \varphi  +  S (\psi_* )  +  {\mathcal N} (\sum\limits_{i=1}^N\eta_i \Phi_i + \varphi ) \right]\\
		&+ \sum\limits_{i=1}^N \left( L[\eta_i \Phi_i ] -\eta_i L[\Phi_i]  \right)    =0 \quad  {\mbox {for}} \quad x \in \R^2 , \\
		d_1 &[\vp , \phi , \alpha ] =0,
	\end{aligned}
	\ee
	where $\hat B_\ve$ is the linear operator
	\be\label{Bhat}
	\hat B_\ve [\phi ]=B_\ve [\phi ] + b_\ve (y)\phi.
	\ee
	We recall that
	$$\rho= \frac {\delta_1}{\ve\mu\sqrt{|\log\ve|}} $$
	where $\delta_1$ is fixed as in \eqref{eta1}. As explained in Section \S \ref{quattro} if $(\vp , \phi , \alpha )$ solves the Problem \eqref{inout}, then the function $\psi_\#$ of the form \eqref{ansatz} solves \eqref{fin}, giving the result in Theorem \ref{teo2}.

	\medskip

	We begin by solving the second equation in \eqref{inout}, assuming that $\alpha $ satisfies the bound in \eqref{defR} and that the following conditions on $\phi$ are met. We assume $\phi$  is given by
	\be \label{phixx}
	\phi = \phi_1 + \phi_2 
	\ee
	where $\phi_1$ solves
	\be\label{eqphi1}
	\Delta_y \phi_1 + e^{\Gamma } \phi_1 + \hat B_\ve [\phi_1 ] + \hat B_\ve[\phi_2] +H(\phi,\vp,r)=d_1 e^{\Gamma }Z_1 + d_0 e^\Gamma Z_0 \quad \mbox{in}\quad B_\rho,
	\ee
	for some constants $d_0$, $d_1$, and
	\be\label{phi2ex}
	\phi_{2}(y) = d_{0}\phi_{2o}(y)\quad \mbox{with}\quad  \phi_{2o}(y):=\left( \frac{4}{3}\frac{|y|^2-1}
	{|y|^2+1}\log(1+|y|^2)-\frac 83\frac{1}
	{|y|^2+1}\right).
	\ee
	A direct computation shows that 
	$$
	\Delta \phi_2 + e^\Gamma \phi_2 + d_0 e^\Gamma Z_0 =0.$$
	Hence $\phi_1 + \phi_2$ solves the first equation in \eqref{inout}.
	For the definition of $Z_i$, we refer to \eqref{defZj}.
	
	\medskip
	Let $\beta \in (0,1)$ and $m >2$. For a function $\phi \in C^{2, \beta} (B_\rho ) $  we define
	\be \label{normin}
	\| \phi \|_* = \| \phi \|_{m-2 , \rho} + \| D_y \phi \|_{m-1 , \rho} + \| D_y^2 \phi \|_{m, \rho, \beta}
	\ee
	where
	\begin{align*}
		\| \phi \|_{m-2, \rho} &= \sup_{y \in B_\rho} |(1+ |y|)^{m-2} \phi (y) |, \quad \\
		\| \phi \|_{m-2, \rho, \beta }
		&= \sup_{y \in B_\rho} |(1+ |y|)^{m-2} \phi (y) | + \sup_{y \in B_\rho} (1+ |y|)^{m-2 + \beta} [\phi]_{B_1(y) \cap B_\rho, \beta}
	\end{align*}
	with
	$$
	[\phi ]_{B_1(y) \cap B_\rho, \beta} 
	= \sup_{y_1 \not= y_2, y_1, y_2 \in B_1(y) \cap B_\rho} {|\phi (y_1) - \phi (y_2) |\over |y_1 - y_2 |^\beta} .$$

	\medskip
	We shall solve the second equation in \eqref{inout} 
	\be \label{nn1} \begin{aligned}
		L[\varphi ] &  + \left(1-\sum\limits_{i=1}^N\eta_i \right)\left[   F'( \psi_* -{\alpha \over 2} |\log \ve| \, |x|^2 ) \varphi  +  S (\psi_* )  +  {\mathcal N} (\sum\limits_{i=1}^N\eta_i \Phi_i + \varphi ) \right]\\
		&+ \sum\limits_{i=1}^N \left( L[\eta_i \Phi_i ] -\eta_i L[\Phi_i]  \right)    =0 \quad  {\mbox {for}} \quad x \in \R^2 
	\end{aligned}
	\ee
	assuming that
	\be \label{ip1}
	\begin{aligned} &\phi= \phi_1 + \phi_2, \quad \phi_1 (y_1 , y_2 ) = \phi_1 (y_1 , - y_2) , \quad
		\\
		& \| \phi_1 \|_*  + |d_1| \leq c \, \ve \mu \sqrt{|\log \ve |},  
		\quad  |d_0| \leq c (\ve \mu \sqrt{|\log \ve |})^{m-1}.
	\end{aligned}
	\ee
	This is the content of the following propostion.
	\medskip
	\begin{prop}\label{outproblem}
		Assume $\phi$ 
		satisfy conditions \eqref{ip1} and $\alpha$ satisfies the bound in \eqref{alpha}.
		Then there  exist constants $\ve_0 >0$, $C>0$, $\sigma \in (0,1)$, $\sigma <m-2$ with $m$ as in the definition of \eqref{normin},   such that for all $\ve \in (0, \ve_0)$ the following holds: There exists a unique solution $\vp = \vp (\phi , \alpha )$, $\varphi \in C^{1,\beta } (\R^2)$ to \eqref{nn1} satisfying the dihedral symmetry \eqref{dihevar} and  such that 
		\be \label{estout1}
		\| (1+ |x|^2)^{-1} \vp \|_\infty \leq C \, \ve^{1+ \sigma}
		\ee
		where $\sigma \in (0,1) $ as in Proposition \ref{prop4}. Moreover, there exists $\bar \ell \in (0,1)$ such that
		\be \label{estout2}
		\| (1+ |x|^2)^{-1} \, [ \vp (\phi^1 , \alpha ) - \vp (\phi^2 , \alpha) ]\|_\infty \leq  \bar \ell \, \, (\| \phi_1^1 - \phi_1^2 \|_* +|d_0^1-d_0^2|)
		\ee
		\be \label{estout3}
		\| (1+ |x|^2)^{-1} \, [ \vp (\phi , \alpha_1 ) - \vp (\phi , \alpha_2 ) ]\|_\infty \leq  \bar \ell \, \, |\alpha_1 - \alpha_2|.
		\ee
		where $\phi^\ell=\phi_1^\ell+d_0^\ell\phi_{2o}$ for $\ell=1,2$.
	\end{prop}
	
	\medskip
	The basic theory to prove this result relies on the solvability and a-priori bounds for the 
	Poisson equation for the operator $L$
	\be
	\label{louter}
	L[\vp]  +  g(x) = 0  \inn \R^2 ,
	\ee
	for a bounded function $g$. Here $L$ is the differential operator in divergence form defined in \eqref{defL}.

	We take functions  $g(x)$ that satisfy the decay condition
	\be \label{normout}
	\| g\|_{\bar \nu}\, :=\, \sup_{x\in \R^2} (1+|x|)^{\bar \nu}|g(x)|\, <\,+\infty \, , \quad {\bar \nu} >2.
	\ee
	\begin{lemma}\label{prop2} (Proposition 6.1 in \cite{guerra-musso}) Assume $g$ satisfies the dihedral symmetry \eqref{dihege}, and $\| g \|_{\bar \nu }<\infty$.
		There exist $C>0$ and a solution $\vp (x)$ to problem $\equ{louter}$, which is of class $C^{1,\beta}(\R^2)$ for any $0<\beta<1$,  that defines a linear operator $\vp = {\mathcal T}^o (g) $ of $g$
		and satisfies the bound
		\be\label{estimate}
		|\vp (x)| \,\le \,     C{ \| g\|_{\bar \nu} }(1+ |x|^2).
		\ee
		Besides, $\varphi$ satisfies \eqref{dihe}. We refer to \eqref{normout} for the definition of the norm $\| \cdot \|_{\bar \nu}$.
	\end{lemma}
	
	\begin{proof}
		The proof is given in \cite{guerra-musso} (Proposition 6.1), an the dihedral symmetry for $\vp$ follows from the fact that the operator $L$ is invariant under rotation.
	\end{proof}
	
	\medskip
	
	We now turn to the proof of Proposition \ref{outproblem}.
	
	\medskip
	\begin{proof}[Proof of Proposition \ref{outproblem}]
		We let 
		$X^o$ be the Banach space of all functions $\vp \in C^{1,\beta}(\R^2)$  satisfying the dihedral symmetry \eqref{dihevar} such that
		\[
		\| (1+ |x|^2)^{-1} \, \vp \|_\infty < +\infty,
		\]
		and let $G$ be
		\be \label{defG}
		\begin{aligned}
			G(\psi, \phi, \alpha)&= \left(1-\sum\limits_{i=1}^N\eta_i \right)\left[   F'( \psi_* -{\alpha \over 2} |\log \ve| \, |x|^2 ) \varphi  +  S (\psi_* )  +  {\mathcal N}  \left(\sum\limits_{i=1}^N\eta_i \Phi_i + \varphi \right) \right]\\
			&+ \sum\limits_{i=1}^N \left( L[\eta_i \Phi_i ] -\eta_i L[\Phi_i]  \right) .
		\end{aligned}
		\ee
		Since $\phi$ and $\alpha $ satisfy \eqref{ip1} and \eqref{alpha} respectively, by construction we have that 
		$  G(\psi, \phi , \alpha)$ satisfies \eqref{dihege} for functions $\vp \in X^o$.
		We formulate the outer equation \equ{nn1} 
		as the fixed point problem in $X^o$
		\begin{align*}
			\vp  =  \TT^o [   G(\vp, \phi , \alpha)   ], \quad \vp \in X^o
		\end{align*}
		where ${\mathcal T}^o$ is defined in Lemma \ref{prop2} and $G$ is given in \eqref{defG}.
		
		The strategy is to use the Contraction Mapping Theorem in the set $\hat{\mathcal B} = \{ \vp \in X^o \, : \, \| (1+ |x|^2)^{-1} \, \vp \|_\infty \, \leq C \ve^{1+\sigma} \}$. We claim that there exist $\sigma >0$,  $C>0$ independent of $\ve $ such that $ \TT^o [   G(\vp, \phi)   ] \in \hat{\mathcal B}$ for any $\phi$ satisfying \eqref{ip1}, provided  $\vp \in \hat{\mathcal B}.$ Let $\vp \in \hat{\mathcal B}.$ 
		Since
		$$
		(1-\sum\limits_{i=1}^N\eta_i ) (x) \not= 0 \quad \Leftrightarrow \quad |M_j^{-1} (x-P_j ) | >{2 \delta_1 \over \sqrt{|\log \ve |}} \quad \forall j,
		$$
		a combination of \eqref{SPsi02} and \eqref{ee2n} gives the existence of $\sigma \in (0,1)$, which can be taken $\sigma <a$, such that
		$$
		\left|(1-\sum\limits_{i=1}^N\eta_i )  S(\psi_*) \right| \leq C_1 {\ve^{1+\sigma} \over 1+ |x|^{\bar \nu}},
		$$
		for some $C_1 >0$. 
		Besides, since 
		$$
		F' (\psi_*-{\alpha \over 2} |\log \ve| \, |x|^2) = \ve^2 \left( \eta' (\psi_* -{\alpha \over 2} |\log \ve| \, |x|^2 ) + \eta (\psi_* -{\alpha \over 2} |\log \ve| \, |x|^2 )\right) \, e^{\psi_* -{\alpha \over 2} |\log \ve| \, |x|^2}
		$$
		the support of $F' (\psi_* -{\alpha \over 2} |\log \ve| \, |x|^2 ) $ is contained in $B(0,1)$ thanks to \eqref{cutout}.
		Arguing as in \eqref{rhsf} we get
		\begin{align*}
			\left| (1-\sum\limits_{i=1}^N\eta_i )  F'( \psi_* -{\alpha \over 2} |\log \ve| \, |x|^2 ) \varphi \right| & \leq \ve^2 \mu^2  {  \, \| (1+ |x|^2)^{-1} \, \vp \|_\infty  \over 1+|x|^{\bar \nu}}  .
		\end{align*}
		On the other hand, from \eqref{nonl} we get that
		\begin{equation}\label{expN}
			|{\mathcal N} (g)| \leq C |F(\psi_* -{\alpha \over 2} |\log \ve| \, |x|^2) | |g|^2; 
		\end{equation}
		hence
		\begin{align*}
			\left| \left(1-\sum\limits_{i=1}^N\eta_i \right)  {\mathcal N}  \left(\sum\limits_{i=1}^N\eta_i \Phi_i + \varphi \right)  \right|& \leq C  \left| \left(1-\sum\limits_{i=1}^N\eta_i \right)  F(\psi_* -{\alpha \over 2} |\log \ve| \, |x|^2)\right| \left( \sum_{i=1}^N | \eta_i \Phi_i |^2 + |\vp |^2 \right) .
		\end{align*}
		Since $\Phi_i (x) = \Phi (Q_i^{-1} x)$, $\Phi (x) = \phi (y)$, with $y= {M^{-1} (x-P_1) \over \ve \mu}$ and $\vp$ satisfying \eqref{ip1}, we get
		$$
		\left| \left(1-\sum\limits_{i=1}^N\eta_i \right)  {\mathcal N}  \left(\sum\limits_{i=1}^N\eta_i \Phi_i + \varphi \right)  \right| \leq
		o(1) {\ve^{1+\sigma} \over 1+ |x|^{\bar \nu}},
		$$
		with $o(1) \to 0$ as $\ve \to 0$. Finally we consider the term $\sum\limits_{i=1}^N \left( L[\eta_i \Phi_i ] -\eta_i L[\Phi_i]  \right) $. Since the operator $L$ is invariant under rotation, it is enough to treat the term $ L[\eta_1 \Phi_1 ] -\eta_1 L[\Phi_1]$. 
		Let
		$$
		\tilde \eta_1 (z) = \eta_1 (Mz + P_1), \quad \tilde \Phi_1 (z) = \Phi_1 (Mz+P_1).$$
		From Proposition \ref{prop01} we obtain
		\begin{align*}
			L[\eta_1 \Phi_1 ] &-\eta_1 L[\Phi_1] = 2 \nabla_z \tilde \eta_1 \nabla_z \tilde \Phi_1 + \tilde \Phi_1 \Delta_z \tilde \eta_1 + B [\tilde \eta_1 \tilde \Phi_1 ] - \tilde \eta_1 B[\tilde \Phi_1]   .
		\end{align*}
		Using the bound in \eqref{normin} for the function $\phi$ and the fact that the above function has compact support, we obtain
		\begin{align*}
			\left|\sum\limits_{i=1}^N \left( L[\eta_i \Phi_i ] -\eta_i L[\Phi_i]  \right) \right| \leq o(1)\,  {\ve^{1+ \sigma} \over 1+ |x|^{\bar \nu}}
		\end{align*}
		for $o(1) \to 0$ as $\ve \to 0$. We check this estimate with the first term 
		$2 \nabla_z \tilde \eta_1 \nabla \tilde \Phi_1 $. Using the bounds \eqref{ip1}, We have
		\begin{align*}
			|&2 \nabla_z \tilde \eta_1 \nabla_z \tilde \Phi_1 | \leq C \sqrt{|\log \ve |} \,\left |\eta' \left({\sqrt{|\log \ve |} \, M^{-1} (x-P_1) \over \delta_1 }  \right)\right| \, |\nabla_z \tilde \Phi_1|\\
			&\leq  C {\sqrt{|\log \ve |} \over \ve \mu} \, \left|\eta' \left({\sqrt{|\log \ve |} \, M^{-1} (x-P_1) \over \delta_1 }  \right)\right|\left( { \| \phi_1 \|_* \over (1+ |{z \over \ve \mu \sqrt{|\log \ve|}}|)^{m-1} }+{c_0 \over (1+ |{z \over \ve \mu \sqrt{|\log \ve|]}}|)}\right) \\
			&\leq C { ( \ve \mu)^{m-1} |\log \ve |^{1 +{m-1 \over 2}} \over 1+ |x|^{\bar \nu} }  .
		\end{align*}
		Since we took $\sigma <m-2$, we obtain the desired estimate. Similar arguments work for the remaining terms.
		
		\medskip
		A combination of these estimates and the a-priori bound \eqref{estimate} in Lemma \ref{prop2} gives that
		$$
		|\TT^o [   G(\vp, \phi , \alpha)   ] (x)| \leq C {\ve^{1+ \sigma} \over 1+ |x|^2}
		$$
		for some $C>0$. Hence $\TT^o [   G(\vp, \phi , \alpha)  ]\in \hat{\mathcal B}$.
		
		Next we prove that $\vp \to \TT^o [   G(\vp, \phi , \alpha)  ]$ is a contraction mapping in $\hat{\mathcal B}$. Take $\vp_1$, $\vp_2 \in \hat{\mathcal B}.$ Then
		\begin{align*}
			| \TT^o [   G(\vp_1 , \phi , \alpha)  ] - \TT^o [   G(\vp_1 , \phi , \alpha )  ] |& \leq C
			\left|\left(1-\sum\limits_{i=1}^N\eta_i \right)   F'( \psi_*  -{\alpha \over 2} |\log \ve| \, |x|^2) \right| \, |\varphi_1 - \vp_2 | \\
			+ C \left|\left(1-\sum\limits_{i=1}^N\eta_i \right)\right| & \, \left| {\mathcal N}  \left(\sum\limits_{i=1}^N\eta_i \Phi_i + \varphi_1  \right) - {\mathcal N}  \left(\sum\limits_{i=1}^N\eta_i \Phi_i + \varphi_2  \right) \right|  
		\end{align*}
		We use \eqref{cutout} and the definition of the cut-off functions $\eta_i$ (see \eqref{eta1}) to get
		\[
		\left|\left(1-\sum\limits_{i=1}^N\eta_i \right)   F'( \psi_* -{\alpha \over 2} |\log \ve| \, |x|^2) \right| \, |\varphi_1 - \vp_2 |
		\leq {o(1) \over 1+ |x|^\nu} \, \| (1+ |x|^2)^{-1} (\vp_1 - \vp_2) \|_\infty, 
		\]
		for $o(1) \to 0 $ as $\ve \to 0$. 
		Moreover from \eqref{nonl}, we write
		\begin{align*}
			{\mathcal N} &\left(\sum\limits_{i=1}^N\eta_i \Phi_i + \varphi_1 \right) - {\mathcal N} \left(\sum\limits_{i=1}^N\eta_i \Phi_i + \varphi_2\right) \\
			&= F \left(\psi_* -{\alpha \over 2} |\log \ve| \, |x|^2+ \sum\limits_{i=1}^N\eta_i \Phi_i + \varphi_1 \right) -  F \left(\psi_* -{\alpha \over 2} |\log \ve| \, |x|^2+ \sum\limits_{i=1}^N\eta_i \Phi_i + \varphi_2 \right) \\ 
			&- F' (\psi_* -{\alpha \over 2} |\log \ve| \, |x|^2) [ \varphi_1 -\vp_2 ] .
		\end{align*}
		Hence we bound
		\begin{align*}
			\left|\left(1-\sum\limits_{i=1}^N\eta_i \right)\right| & \, \left| {\mathcal N}  \left(\sum\limits_{i=1}^N\eta_i \Phi_i + \varphi_1  \right) - {\mathcal N}  \left(\sum\limits_{i=1}^N\eta_i \Phi_i + \varphi_2  \right) \right| \\
			&\leq  \left|\left(1-\sum\limits_{i=1}^N\eta_i \right)\right| \, \left| F' \left(\psi_* -{\alpha \over 2} |\log \ve| \, |x|^2 + \sum\limits_{i=1}^N\eta_i \Phi_i \right)\right| \, |\vp_1 - \vp_2 | \\
			& + \left|\left(1-\sum\limits_{i=1}^N\eta_i \right)\right| \, | F' (\psi_* -{\alpha \over 2} |\log \ve| \, |x|^2)| \, |\vp_1 - \vp_2 |.
		\end{align*}
		Under assumptions \eqref{ip1} on $\phi$ we have that
		$\psi_* + \sum\limits_{i=1}^N\eta_i \Phi_i = \psi_* (1+ o(1))$ uniformly on $\R^2$, as $\ve \to 0$. This fact and an argument similar to the previous one give
		\begin{align*}
			\left|\left(1-\sum\limits_{i=1}^N\eta_i \right)\right| & \, \left| {\mathcal N}  \left(\sum\limits_{i=1}^N\eta_i \Phi_i + \varphi_1  \right) - {\mathcal N}  \left(\sum\limits_{i=1}^N\eta_i \Phi_i + \varphi_2  \right) \right| \\
			&\leq {o(1) \over 1+ |x|^{\bar \nu}} \, \| (1+ |x|^2)^{-1} (\vp_1 - \vp_2) \|_\infty.
		\end{align*}
		A combination of these estimates and the a-priori bound \eqref{estimate} in Lemma \ref{prop2} gives that
		$$
		|\TT^o [   G(\vp_1, \phi, \alpha)   ] - \TT^o [   G(\vp_2, \phi, \alpha )   ] | \leq o(1) { \| (1+ |x|^2)^{-1} (\vp_1 - \vp_2) \|_\infty \over 1+ |x|^2}
		$$
		with $o(1) \to 0$ as $\ve \to 0$. Hence $\vp \to \TT^o [   G(\vp, \phi, \alpha)  ]$ is a contraction mapping in $ \hat{\mathcal B}$.

		\medskip
		Applying the contraction mapping Theorem we obtain that there exists a unique
		$\vp \in \hat{\mathcal B}$ solution to \eqref{nn1}.
		
		\medskip
		To prove \eqref{estout2}, let $\vp^i = \vp (\phi_i , r)$, $i=1,2$, for some fixed $r$ satisfying \eqref{defR}. The $\bar \vp = \vp^1 - \vp^2$ satisfies
		\begin{align*}
			L[\bar \varphi ] &  + \left(1-\sum\limits_{i=1}^N\eta_i \right)\left[   F'( \psi_* -{\alpha \over 2} |\log \ve| \, |x|^2) \bar \varphi   +  {\mathcal N} \left(\sum\limits_{i=1}^N\eta_i \Phi_i^1 + \varphi^1 \right) - {\mathcal N} \left(\sum\limits_{i=1}^N\eta_i \Phi_i^2 + \varphi^2 \right) \right]\\
			&+ \sum\limits_{i=1}^N \left( L[\eta_i (\Phi_i^1 - \Phi_i^2) ] -\eta_i L[(\Phi_i^1 - \Phi_i^2)]  \right)    =0 \quad  {\mbox {for}} \quad x \in \R^2 
		\end{align*}
		where $\Phi_i^\ell (x) = \Phi^\ell (Q_i^{-1} x)$, $\Phi^\ell (x) = \phi^\ell (y) $ for $y= {M^{-1} (x-P_1) \over \ve \mu }$.
		This problem can be solved using a fixed point argument, as before, which gives the desired estimate \eqref{estout2}. Similarly, we get \eqref{estout3}.
		This concludes the proof.

	\end{proof}

	\medskip

	To solve equation \eqref{eqphi1} in $\phi_1$, we consider the problem 
	\be
	\Delta \phi  +  e^{\Gamma} \phi  + \hat B_\ve [\phi]   + h(y)  = d_1 e^{\Gamma} Z_1 + d_0 e^{\Gamma} Z_0 \inn B_\rho (0)
	\label{eee} \ee
	where $\rho = {\delta_1 \over \ve \mu \sqrt{|\log \ve |}}$.

	For a function $h$ defined in $B_\rho (0)$, numbers $m>2$ and $\beta \in (0,1)$ we denote by $\|h\|_{m,\beta,\rho}$ 
	$$
	\|h \|_{m,\beta, \rho} =  \sup_{y\in B_\rho (0) } (1+|y|)^{m+\beta}[ h]_{B(y,1)\cap B_\rho (0)}   +     \|h \|_{m, \rho}
	$$
	where
	\begin{align*}
		\|h \|_{m, \rho} &=  \sup_{y\in  B_\rho (0) } | (1+|y|)^{m} h(y)|] .\\    
	\end{align*}
	In Appendix \ref{AppeB} we prove:
	\begin{lemma}\label{prop1}  Assume $h$ satisfies 
		\be \label{even}
		h(y_1 , y_2) = h(y_1 , -y_2), \quad \forall y \in \RR^2
		\ee
		and $\| h \|_{m,\beta, \rho} <\infty$, for some $m>2$,  $\beta \in (0,1)$. There is $C>0$ such that for all $\ve >0$ small
		Problem $\equ{eee}$ has a solution $\phi = {T}[h]$ for certain scalars $d_{i}= d_{i}[h]$, $i=0,1$, that defines a linear operator of $h$ and satisfies
		\begin{align*} 
			\| \phi\|_{*} \ \le\ C \|h \|_{m,\beta, \rho} ,
		\end{align*}
		where $\| \cdot \|_*$ is defined in \eqref{normin}. 
		In addition, the $c_i$'s can be estimated as
		\begin{align*} 
			d_{0}[h]\, = & \, \gamma_0\int_{B_\rho }  h Z_0  + O( (\ve \mu \sqrt{|\log \ve |})^{m-2})  \|h \|_{m,\beta, \rho} , \\    d_{1}[h]\, = &\, \gamma_1\int_{B_\rho}  h Z_1  + O( \ve \mu \sqrt{|\log \ve |})  \|h \|_{m,\beta, \rho}, 
		\end{align*}
		where
		$\gamma_j^{-1} = \int_{\R^2} e^{\Gamma} Z_j^2$, $j=0,1$.
	\end{lemma}
	
	Let us consider the operator $\vp [\phi] $ defined in Proposition \ref{outproblem}. Then we find a solution to the first equation in \eqref{inout} of the form $\phi = \phi_1 + \phi_2$ as in \eqref{phixx},  if $\phi_1$ solves  
	\be\label{eqphi2}
	\Delta_y \phi_1 + e^{\Gamma } \phi_1 + \hat B_\ve [\phi_1] + \hat B_\ve [\phi_2] +H(\phi,\vp [\phi] ,r)=d_1 e^{\Gamma }Z_1 + d_0 e^\Gamma Z_0 \quad \mbox{in}\quad B_\rho,
	\ee
	where $\phi_2$ is given by \eqref{phi2ex}, that is 
	\be\label{phi2ex2}
	\phi_{2}(y) = d_{0}\,\phi_{2o}(y)\quad \mbox{with}\quad  \phi_{2o}(y):=\left( \frac{4}{3}\frac{|y|^2-1}
	{|y|^2+1}\log(1+|y|^2)-\frac 83\frac{1}
	{|y|^2+1}\right),
	\ee
	recalling that  
	$
	\Delta \phi_2 + e^\Gamma \phi_2 + d_0 e^\Gamma Z_0 =0.$ Here
	by Lemma \ref{prop1}
	\begin{equation}\label{c0c1}
		d_0=d_0[H(\phi,\vp [\phi] ,\alpha)+\hat B_\ve [\phi_2]],\quad \mbox{and}\quad d_1=d_1[H(\phi,\vp [\phi] ,\alpha)+\hat B_\ve [\phi_2]].
	\end{equation}
	See \eqref{in-out-1} and \eqref{Bhat} for the definitions of $H$ and $\hat B_\ve$ respectively.
	\medskip

	\begin{prop}\label{propphi1}
		Assume $\alpha$ satisfies the bound in \eqref{alpha} and let $a>0$ be given by Proposition \ref{prop4}. Choose  $m$ such that $2<m<2+a$.  Then there exists a unique solution $(\phi_1,d_0) $ to \eqref{eqphi2}, \eqref{phi2ex2}, and \eqref{c0c1}, that satisfies
		\be \label{esttt11}
		\| \phi_1 \|_* \leq C \, \ve \, \mu \, \sqrt{|\log \ve |},
		\quad |d_0|\leq C (\ve \, \mu \, \sqrt{|\log \ve|})^{m-1}.
		\ee
		Besides, there is a continuous dependence of $\phi_1$ and $d_0$ on $r$. See \eqref{normin} for the definition of $\| \cdot \|_*$.
	\end{prop}

	\begin{proof}
		We state Problem  \eqref{eqphi2} as a fixed point equation for $\phi_1$
		$$
		(\phi_1, d_0) = \overline{{\mathcal A}} (\phi_1,d_0) , 
		$$
		where
		$$
		\phi_1= { T} \left( H (\phi_1+\phi_2, \vp[\phi_1+\phi_2], \alpha )+\hat B_\ve[\phi_2] \right)
		$$
		$$
		d_0=\gamma_0\int\limits_{\RR^2} \left( H (\phi_1+\phi_2, \vp[\phi_1+\phi_2], \alpha )+\hat B_\ve[\phi_2] \right) \, Z_0
		$$
		and 
		$
		\phi_2(y)=d_0 \phi_{2o}(y)
		$, given by \eqref{phi2ex2}.
		We recall the definition of $H,$  
		$$
		H(\phi , \varphi , \alpha ) =  \ve^2 \mu^2  {\mathcal N} \left(\sum\limits_{i=1}^N\eta_i \Phi_i + \vp \right) +\ve^2 \mu^2 \, F'(\psi_* -{\alpha \over 2} |\log \ve| \, |x|^2) \varphi +    \ve^2 \mu^2 S(\psi_*) 
		$$
		as given in
		\eqref{in-out-1}. Using \eqref{rmk1}, the dihedral symmetry \eqref{dihe} of $\vp$,  assumptions \eqref{ip1} on $\phi$, we have, by uniqueness of $\vp$ in equation \eqref{nn1} (Proposition \ref{outproblem}), that $\tilde \vp(y)=\vp(M\ve y+P_1)$ satisfies also \eqref{even},  and consequently we obtain that $H(\phi, \vp, r)$ satisfies \eqref{even}. On the other hand, since $\phi_2$ is a function of $|y|^2$, and the form of $\hat B_\ve$ given in \eqref{Bhat}, we have that $\hat B_\ve[\phi_2]$ satisfies 
		\eqref{even}.
		
		\medskip
		Let ${\mathcal B} = \{ (\phi_1,d_0) \, : \phi_1(y_1, y_2) = \phi_1 (y_1 , - y_2), \, \, \| \phi_1 \|_* \leq c \, \ve \, \mu \, \sqrt{|\log \ve|},\, \, |d_0|\leq c (\ve \, \mu \, \sqrt{|\log \ve|})^{m-1}\}$, where $\| \cdot \|_*$ is defined in \eqref{normin}. From \eqref{SPsi01a} and Proposition \ref{prop4}, we get that
		$$
		\| \ve^2 \mu^2 S(\psi_* ) \|_{m, \beta , \rho} \leq C \, \ve \, \mu \, \sqrt{|\log \ve |}.
		$$
		From the form of $F(\psi_* -{\alpha \over 2} |\log \ve|\, |x|^2)$ in \eqref{defF}, the estimate \eqref{rhsf}, and the form of ${\mathcal N}$ in\eqref{nonl},  we find for some $a\in (0,1)$,
		\be\label{f1b}
		\begin{aligned}
			|H(\phi_1 + \phi_2 , \varphi , \alpha ) | &\leq  |\tilde E|+
			{8 \over 1+ |y|^{2}} \left({1\over (1+ |y|^{2}) } +  {C \ve\mu \sqrt{|\log \ve |} \over  (1+ |y|^{a})}\right)|\varphi| \\
			&+ {C \over (1+ |y|)^4} \left(|\varphi|^2 + \sum\limits_{i=1}^N(|\eta_i \Phi_{i,1}|^2 + |\eta_i \Phi_{i,2}|^2) \right),
		\end{aligned}
		\ee
		where $\Phi_{i,\ell} (x) = \phi_\ell(y)$ with  $y= {M^{-1} (Q_i^{-1}x-P_1) \over \ve \mu }$ for $\ell=1,2$. Here 
		\[
		|\tilde E | \leq C  { \ve\mu \sqrt{|\log \ve |} \over (1+ |y|^{2+a})} . 
		\]
		Using the assumptions \eqref{estout1} on $\varphi$, and 
		the definiton of ${\mathcal B}$ to estimate $\Phi_{i,1}$ and $\Phi_{i,2}$, we get by \eqref{f1b} that
		\be \label{f2}
		\|H( \phi_1 + \phi_2 , \varphi , \alpha) \|_{m, \beta, \rho} \leq C \ve\mu \sqrt{|\log \ve|},
		\ee
		since $m<2+a$. From \eqref{Bexpr}, \eqref{biest}, and \eqref{By}, we get
		\be\label{f2b}
		\begin{aligned}
			| \hat B_\ve [\phi_{2}] | &\leq C\ve \mu \left(\frac 1{\sqrt{|\log\ve|}}(|D_y\phi_{2}|+|y| |D^2_y\phi_{2}|) + {\sqrt{|\log \ve |} \over 1+ |y|^{2+a}} |\phi_{2}|\right)
		\end{aligned}
		\ee
		and using again that $m<2+a$, we have
		\be \label{f3}
		\begin{aligned}
			\|\hat B_\ve&[\phi_{2}]  \|_{m, \beta, \rho} \leq C (|\log\ve|)^{-1}(\ve\mu \sqrt{|\log \ve |})^{2-m} d_0\\
			& + C  (\ve\mu \sqrt{|\log \ve |}) |\log \ve | d_0 
			\leq   C (\ve\mu \sqrt{|\log \ve |})^{2-m} d_0 \leq C \ve\mu \sqrt{|\log \ve |}.
		\end{aligned}
		\ee
		Hence, from Lemma \ref{prop1} we obtain that
		$$
		\|   { T} (  H (\phi_1+\phi_2, \vp[\phi_1+\phi_2], \alpha )+\hat B_\ve[\phi_2])
		\|_* \leq C_1 \, \ve \mu  \sqrt{|\log \ve|}$$
		for some $C_1 >0$. Using \eqref{f2b}, we have that $\int_{B_\rho} \hat B_\ve[\phi_2]Z_0= O\left(\frac{d_0}{|\log\ve|}\right)$ and so  similarly to the proof of Proposition 7.1 in \cite{guerra-musso}, we have 
		\begin{equation}\label{HBZ0}
			\int_{B_\rho} (H (\phi_1+\phi_2, \vp[\phi_1+\phi_2], \alpha  )+\hat B_\ve[\phi_2])Z_0= O\left(\frac{d_0}{|\log\ve|}\right).
		\end{equation}
		By using the formula in Lemma \ref{prop1}, we have 
		$$
		|d_0|\leq C (\ve\mu \sqrt{|\log \ve |})^{m-2}\| H (\phi_1+\phi_2, \vp[\phi_1+\phi_2], r )+\hat B_\ve[\phi_2]\|_{m,\beta,\rho}+O\left(\frac{d_0}{|\log\ve|}\right).
		$$
		So we get by \eqref{f2} and \eqref{f3} that
		$$
		|d_0|\leq C (\ve\mu \sqrt{|\log \ve |})^{m-1}.
		$$
		We can choose the constant $c$ in the definition of ${\mathcal B}$ in such a way that $\overline{\mathcal A} $ sends elements in ${\mathcal B}$ to ${\mathcal B}$. 
		
		\medskip
		Take now $\phi_1^1$, $\phi_1^2 \in {\mathcal B}$, with respective $\phi_2^1$, $\phi_2^2$. Then using \eqref{nonl}, we have
		\begin{align}\nonumber
			H (\phi^1, \vp[\phi^1], \alpha) &- H (\phi^2, \vp[\phi^2], \alpha )  =  \ve^2 \mu^2  {\mathcal N} \left(\sum\limits_{i=1}^N\eta_i \Phi_i^1 + \vp [\phi^1]\right) \\
			&-  \ve^2 \mu^2  {\mathcal N} \left(\sum\limits_{i=1}^N\eta_i \Phi_i^2 + \vp [\phi^2] \right) + \ve^2 \mu^2 \, F'(\psi_* -{\alpha \over 2} |\log \ve|\, |x|^2) [ \varphi [\phi^1] - \varphi [\phi^2] ] \label{difH}
		\end{align}
		where $\Phi_i^\ell (x) = \Phi^\ell (Q_i^{-1} x)$, $\Phi^\ell (x) = \phi^\ell (y) $ for $y= {M^{-1} (x-P_1) \over \ve \mu }$.
		Also we have
		\begin{align*}
			\hat B_\ve[\phi_2^1]-\hat B_\ve[\phi_2^2]=
			(d_0^1-d_0^2)\hat B_\ve[\phi_{2o}]
		\end{align*}
		where $d_0^\ell=\gamma_0\int\limits_{\RR^2}(
		H (\phi^\ell, \vp[\phi^\ell], \alpha )+\hat B_\ve[\phi_2^\ell]) Z_0$ for $\ell=1,2.$ As in \eqref{f3}, we have
		\begin{equation}\label{Bhatdif}
			\|\hat B_\ve[\phi_2^1]-\hat B_\ve[\phi_2^2]\|_{m,\beta,\rho}\leq C|d_0^1-d_0^2| (|\log\ve|)^{-1}(\ve\mu\sqrt{|\log\ve|})^{2-m}.
		\end{equation}
		By \eqref{difH}, using \eqref{expN} and the bounds $\|\phi_1^\ell\|\leq c (\ve\mu\sqrt{|\log\ve|})$,  $|d_0^\ell|\leq c(\ve\mu\sqrt{|\log\ve|})^{m-1}$, and \eqref{estout1}, we have  
		\begin{align} \nonumber
			\|  H (&\phi^1, \vp[\phi^1], \alpha )
			-H (\phi^2, \vp[\phi^2], \alpha )\|_{m,\beta,\rho}\leq \\ \label{Hdif}
			&C(\|\phi_1^1-\phi_1^2\|_*^2+|d_0^1-d_0^2|^2+\|(1+|x|^2)^{-1}(\varphi^1-\varphi^2)\|_\infty^2) \\ \nonumber
			&+C(\ve\mu\sqrt{|\log\ve|})(\|\phi_1^1-\phi_1^2\|_*+|d_0^1-d_0^2|+\|(1+|x|^2)^{-1}(\varphi^1-\varphi^2)\|_\infty). 
		\end{align}
		Then, by  Lemma \ref{prop1}, and \eqref{HBZ0}, together with \eqref{Hdif} and \eqref{Bhatdif}, we get 
		\begin{align*}
			|d_0^1-&d_0^2|\leq C (\ve\mu \sqrt{|\log \ve |})^{m-2}\|  H (\phi^1, \vp[\phi^1], \alpha )
			-H (\phi^2, \vp[\phi^2], \alpha  )+\hat B_\ve[\phi_2^1]- \hat B_\ve[\phi_2^2]\|_{m,\beta,\rho} \\
			&\qquad +O\left(\frac{|d_0^1-d_0^2|}{|\log\ve|}\right) \\
			\leq & C\left(\frac{|d_0^1-d_0^2|}{|\log\ve|}+(\ve\mu\sqrt{|\log\ve|})^{m-1} 
			(\|\phi_1^1-\phi_1^2\|_*+|d_0^1-d_0^2|+\|(1+|x|^2)^{-1}(\varphi^1-\varphi^2)\|_\infty)\right).
		\end{align*}
		On the other hand by Lemma \ref{prop1}, we get
		$$
		\|\phi^1_1-\phi_1^2\|_*\leq C \|  H (\phi^1, \vp[\phi^1], \alpha )
		-H (\phi^2, \vp[\phi^2], \alpha  )+\hat B_\ve[\phi_2^1]- \hat B_\ve[\phi_2^2])\|_{m,\beta,\rho}
		$$
		and using \eqref{Bhatdif}, \eqref{Hdif}, and \eqref{estout2}, we can show that $\overline{{\mathcal A}}$ is a contraction mapping in ${\mathcal B}$. This gives the existence and uniqueness of a solution to \eqref{eqphi2} satisfying \eqref{esttt11}. 
	\end{proof}
	
	\medskip
	We are now in a position to conclude with
	
	\begin{proof}[\it Proof of Theorem \ref{teo2}.] \  The result of Theorem \ref{teo2} follows from solving 
		\eqref{P}. Problem \eqref{P}
		admits a solution of the form
		$$
		\psi= \psi_* + \psi_\#
		$$
		provided that $\psi_\#$ solves \eqref{fin} and it is smaller than $\psi_*$. A function $\psi_\#$ of the form \eqref{ansatz} solves \eqref{fin} if $(\vp, \phi, \alpha )$ satisfy \eqref{inout}. So far we have proven the existence of 
		$\vp$ and $\phi$ solving the first two equations in \eqref{inout}. Besides $\varphi$ satisfies the bound \eqref{estout1}, while $\phi$ has the form $\phi= \phi_1 + d_0\phi_{2o}$ as in \eqref{phixx}-\eqref{eqphi1}-\eqref{phi2ex}. Besides $(\phi_1,d_0)$ satisfies \eqref{esttt11}.
		
		The final step in our argument is to solve the third equation in \eqref{inout}: find $\alpha$ satisfying \eqref{alpha} such that
		$$
		d_1 [ \vp , \phi , \alpha ]=0.
		$$
		Consider the first equation in \eqref{inout}, multiply it against $Z_1$ and integrate over the ball $B_\rho (0)$. We obtain
		\[
		\begin{aligned}
			d_1 \, &\int_{B_\rho} e^{\Gamma} Z_1^2 dy= I + II \\
			I&=  \int_{B_\rho} H(\phi,\vp,\alpha )  Z_1 \, dy \\
			II&=  \int_{B_\rho} \left( \Delta_y \phi + e^{\Gamma } \phi \right) Z_1 \, dy + \int_{B_\rho} \left( B_\ve [\phi ] + b_\ve (y) \phi \right) Z_1 \, dy.
		\end{aligned}
		\]
		We claim that
		\be \label{esI}
		\begin{aligned}
			I&=  2 \ve \mu \sqrt{|\log \ve |} \left( {  r \over  h^2}  - { (N-1) \over r} -{\alpha r \over 2} \right) \left( \int_{\R^2}  U y_1 Z_1 \right) 
			\\
			&+
			\ve \mu  \frac{\log|\log\ve|}{\sqrt{|\log\ve|}}
			\, {\bf Y}_1 (\alpha )
		\end{aligned} 
		\ee
		and
		\be \label{esII}
		II= (\ve \mu )^{2-a}  \, {\bf Y}_1 (\alpha)
		\ee
		where ${\bf Y}_1(\alpha )$ denotes a explicit smooth function,  whose explicit expression changes from line to line. It is uniformly bounded  as $\ve \to 0$  for values of $\alpha$ satisfying   \eqref{alpha}.    
		In \eqref{esII}, the number $a \in (0,{1\over 2})$ is fixed and small.
		
		\medskip
		Assume the validity of \eqref{esI} and \eqref{esII}. By continuity the equation
		$$
		d_1 (\vp , \phi , \alpha )= 0
		$$
		has a solution $r$ that satisfies
		$$
		\alpha= 2 \left( {1\over h^2} - {N-1 \over r^2} \right)   + O\left(\frac{\log|\log\ve|}{|\log\ve|}\right)$$
		as $\ve \to 0$. Thus $\alpha $ satisfies \eqref{alpha}.
		This proves the existence of the solutions predicted by Theorem \ref{teo2}. 
		
		\medskip
		In the rest of the proof we denote with ${\bf Y} (\alpha)$ a generic smooth function of $\alpha $, which is bounded as $\ve \to 0$ for $\alpha$ satisfying \eqref{alpha}.
		
		\medskip
		{\it Proof of \eqref{esI}.} \ \ We recall that
		$B_\rho (0)$ is the ball centered at $0$ of radius $\rho$, with 
		$\rho= \frac {\delta_1}{\ve\mu\sqrt{|\log\ve|}} $ and that 
		$$
		H(\phi, \vp, \alpha )  =  \ve^2\mu^2 {\mathcal N}  \left(\sum\limits_{i=1}^N\eta_i \Phi_i + \vp \right) + 
		(e^{\Gamma(y)}+b_\ve(y))\varphi + \tilde E(y), $$
		with $b_\ve$ given by \eqref{defbe}, where  $\tilde E(y)= \ve^2 \mu^2 S[\psi_* ] (P_1 + \ve \mu My).$
		\smallskip
		
		Using the result in Proposition \ref{prop4}, \eqref{ee2n} and \eqref{SPsi01a} we get
		\begin{align*}
			\int_{B_\rho} \tilde E(y) \, y_1 \, Z_1 dy&=
			\ve \mu \int_{B_\rho} y_1 U(y)  \left(c_{1} \Gamma_0 (y) + {\mathcal A} (r) \right) \, Z_1 dy + \ve^2 \mu^2 {\bf Y} (\alpha)\\
			&=  2 \ve \mu \sqrt{|\log \ve |} \left( {  r \over  h^2}  - { (N-1) \over r} -{\alpha r \over 2} \right) \left( \int_{\R^2}  U y_1 Z_1 \right) 
			\\
			&+
			\ve \mu  \frac{\log|\log\ve|}{\sqrt{|\log\ve|}}
			\, {\bf Y}_1 (\alpha)
		\end{align*}
		We recall the definition of ${\mathcal A} (\alpha)$ in \eqref{defA}.
		
		\medskip
		On the other hand, we have
		\begin{align*}
			\left| \ve^2 \mu^2 {\mathcal N}\left  (\sum\limits_{i=1}^N\eta_i \Phi_i + \varphi \right)\right| & \leq C  \ve^2 \mu^2  F(\psi_* -{\alpha \over 2} |\log \ve|  |x|^2)| \left( \sum_{i=1}^N | \eta_i \Phi_i |^2 + |\vp |^2 \right) \\
			&\leq C U(y) \left( \sum_{i=1}^N | \eta_i \Phi_i |^2 + |\vp |^2 \right), \\
			|\ve^2 \mu^2 F' (\psi_* -{\alpha \over 2} |\log \ve|\, |x|^2) \varphi  | &\leq C U(y)|\varphi|.
		\end{align*}
		From the control we have on $\phi$ and $\vp$ from \eqref{esttt11}  and \eqref{estout1} we conclude that
		\begin{align*}
			\int_{B_\rho} \ve^2 \mu^2 {\mathcal N}  \left(\sum\limits_{i=1}^N\eta_i \Phi_i + \varphi \right)  Z_1 \, dy &=  (\ve \mu)^{2-\alpha} {\bf Y} (\alpha )
		\end{align*}
		for some $\alpha >0$ small. This concludes the proof of \eqref{esI}.
		
		\medskip 
		{\it Proof of \eqref{esII}.} \ \ Since $\Delta Z_1 + e^\Gamma Z_1$,  integration by parts gives
		\begin{align*}
			\int_{B_\rho} (\Delta \phi + e^\Gamma \phi ) Z_1 \, dy &= \int_{\partial B_\rho} \left( {\pp Z_1 \over \pp \nu } \phi - {\pp \phi \over \pp \nu } Z_1\right).
		\end{align*}
		A direct computation gives
		$$
		\left| \int_{\partial B_\rho} \left( {\pp Z_1 \over \pp \nu } \phi - {\pp \phi \over \pp \nu } Z_1\right) \right| \leq C (\ve \mu \sqrt{|\log \ve|})^{m-1} \, \| \phi_1 \|_*+ (\ve \mu \sqrt{|\log \ve|})|d_0|.
		$$
		An inspection to the form of the operator $B$ in Proposition \ref{prop01} gives that, in $B_\rho$, $B_\ve$ has the form
		\be\label{By}
		\begin{aligned}
			B_\ve &=
			\left( -2{R h \over  (h^2+ R^2)^{3/2}} \, \ve \mu y_1 + O(|\ve \mu y|^2 )\right) \pp_{y_1 y_1} + O(|\ve \mu y|^2 ) \pp_{y_2 y_2} \\
			&-\left( 2{R \over h\sqrt{h^2+ R^2} } \ve \mu y_2 + O(|\ve \mu y|^2) \right) \pp_{y_1 y_2}\\
			&
			- \left( {\ve \mu R \over h\sqrt{h^2+ R^2} } \left({2h^2\over h^2+ R^2} +1 \right)  + O((\ve \mu)^2|y|) \right) \pp_{y_1 } \\
			&- \left(\frac{\ve^2 \mu^2y_2}{h^2+R^2} \left({2 h^2\over h^2+ R^2} +1 \right)+O((\ve \mu)^3|y|^2)\right) \pp_{y_2 },
		\end{aligned}
		\ee 
		as $\ve \to 0$. Hence 
		$$
		\left|\int_{B_\rho} B_\ve [\phi] Z_1 \, dy\right| \leq C {\ve \mu \over \sqrt{|\log \ve |}} \, [\| \phi_1 \|_*+|d_0|].$$
		Also, using \eqref{biest} we get
		$$
		\left|\int_{B_\rho}b_\ve(y) \phi  Z_1 \, dy\right| \leq C \ve \mu  \sqrt{|\log \ve |} \, [\| \phi \|_* +|d_0|]. $$
		Using \eqref{esttt11} we get \eqref{esII}. This concludes the proof of Theorem \ref{teo2}.
	\end{proof}

	\appendix
	
	\section{}\label{AppeA}
	
	\begin{proof}[Proof of Proposition \ref{even-ness}.]  For each angle $\theta_j$ the corresponding symmetric angle respect to $x_1$ is given by the $\theta_i=2\pi-\theta_j$ , so we have $i=N-j+2$ and $Q_j=Q_i^{-1}=Q_i^{T}.$ Considering the matrix  $J=\left(\begin{array}{cc} 1 & 0 \\ 0 & -1 \end{array}\right)$, which satisfies $J^{-1}=J,$ we have  $Q_j^{-1}=JQ_i^{-1}J$. Define now $\bar z=J z$, we need to prove that $\Psi_0(\bar z)=\Psi_0(z)$. By definition \eqref{defpsi2}, we have $\Psi_{ \ve \mu}(z)=\Psi_{ \ve \mu}(J z)=\Psi_{ \ve \mu}(\bar z)$. On the other hand, using that $P_i=JP_j$, $JM
		J=M$, and $JM^{-1}
		J=M^{-1}$, we get
		\begin{align*}
			\Psi_{ \ve \mu}(M_j^{-1}(P_1-P_j+M\bar z)) 
			&=\Psi_{ \ve \mu}(M^{-1}Q_j^{-1}(P_1-P_j+M\bar z)) \\
			&= \Psi_{ \ve \mu}(M^{-1}JQ_i^{-1}(P_1-P_i+J M J z)) \\
			&= \Psi_{ \ve \mu}(JM^{-1}JQ_i^{-1}(P_1-P_i+M z)) \\
			&= \Psi_{ \ve \mu}(M^{-1}Q_i^{-1}(P_1-P_i+M z)) \\
			&= \Psi_{ \ve \mu}(M_i^{-1}(P_1-P_i+M z)).
		\end{align*}
		Then
		$$
		\sum_{j=1}^N \Psi_{ \ve \mu}(M_j^{-1}(P_1-P_j+M\bar z))=\sum_{i=1}^N \Psi_{ \ve \mu}(M_i^{-1}(P_1-P_i+Mz))
		$$
		which conclude the proof of \eqref{symmetryPsi}.
	\end{proof}

	\begin{prop}\label{proppsi}
		\begin{eqnarray}\label{sumpsi1}
			\sum_{j=2}^N \Psi_{ \ve \mu}(M_j^{-1}(P_1-P_j))= \sum_{j\neq i}^N \Psi_{ \ve \mu}(M_j^{-1}(P_i-P_j))
		\end{eqnarray}
	\end{prop}
	\begin{proof} If we consider for  $\theta_i$,  the corresponding symmetric angle $\theta_k,$ (with $k=N-i+2$), then 
		$Q_{k}^{-1}=Q_i$. Recalling that $M_j^{-1}=M^{-1}Q_j^{-1}$ and noting that $M_j^{-1}P_j=M_1^{-1}P_1$ for any $j$, we obtain
		$$
		\Psi_{ \ve \mu}(M_k^{-1}(P_1-P_k))= \Psi_{ \ve \mu}(M_1^{-1}(P_i-P_1))
		$$
		so the $j=k$ term  in the left sum in \eqref{sumpsi1} is equal to the first term in the sum on the right, moreover since  
		\begin{equation}\label{Qidentity1}
			Q_{k+m-1}^{-1}=Q_m^{-1}Q_i \quad m=1,\ldots,N-k+1=i-1
		\end{equation}
		then from the  $j=k$ term till $j=N$ term are equal to the corresponding $j=1$ term  till $j=i-1$ term.  In addition we have 
		\begin{equation}\label{Qidentity2}
			Q_{m+1}^{-1}=Q_{i+m}^{-1}Q_i\quad m=1,\ldots,N-i=k-2 
		\end{equation}
		with this from $j=2$ to $j=k-1$ the terms are equal to the $j=i+1$ till $j=N$ terms.  
		
		The identities \eqref{Qidentity1} and \eqref{Qidentity2} follow from
		$ Q_{i_1}Q_{i_2}=Q_{i_1+i_2-1} $
		which follows from direct calculations, and noting in \eqref{Qidentity1} that $Q_i=Q_k^{-1}$.
	\end{proof}
	
	\begin{prop}\label{sumerrorrot} For the rotation matrix $Q_i$ defined in \eqref{defQj}, we have 
		\[ 
		g(x')=g(x)\quad\mbox{where}\quad x'=Q_ix,\quad\mbox{for $i=1,\ldots,N$}
		\]
	\end{prop}
	\begin{proof} We write $g= g_1 + g_2$, where
		$$
		g_1(x)= \eta_0 (|x|) \sum_{j=1}^N \, 
		\,    E(M_j^{-1}(x-P_j)), \quad g_2 (x)=  
		\sum_{j=1}^N 
		\left[  \, L( \eta_{0} \psi_j ) - \eta_0 L(\psi_j)   \right].
		$$
		We first prove that
		$$
		\sum_{j=1}^N \,  E(M_j^{-1}(x-P_j))=\sum_{j=1}^N \,  E(M_j^{-1}(x'-P_j)). 
		$$
		In fact, we first  note that 
		$$
		E(M_i^{-1}(Q_ix-P_i))=E(M_{1}^{-1}(x-P_{1})). 
		$$
		Using a similar argument as in Proposition \ref{proppsi}, we have 
		\begin{align*}
			E(M_j^{-1}(Q_ix-P_j))&=E(M^{-1}Q_j^{-1}Q_ix-M_j^{-1}P_j)\\
			&=E(M^{-1}Q_{N-i+1+j}^{-1}x-M_{N-i+1+j}^{-1}P_{N-i+1+j}) \\
			&=E(M_{N-i+1+j}^{-1}(x-P_{N-i+1+j}))
		\end{align*}
		for $j=1,\ldots,i-1$. ($2\leq i\leq N$). Similarly
		$$
		E(M_j^{-1}(Q_ix-P_j))=E(M_{j-i+1}^{-1}(x-P_{j-i+1})) 
		$$
		for $j=i+1,\ldots,N$ ($1\leq i\leq N-1$). 
		Clearly $\eta_0(|x|)=\eta_0(|x'|)$, so $g_1(x)=g_1(x')$. Now by rotacional invariance of
		$L$ we have 
		\begin{align*}
			[ L( \eta_{0} \psi_j ) - \eta_0 L(\psi_j) ] (x) = [L (\eta_0 \tilde\psi_{j} ) -  \eta_0 L (\tilde \psi_{j})] (x'), \quad \mbox{with}\quad  \tilde \psi_{j}(x)=
			\psi_{j}(x')
		\end{align*}
		therefore $g_2(x)=g_2(x')$ and the proposition is proven.
	\end{proof}
	
	\begin{proof}[Proof of \eqref{tildeg}] \ \ 
		As above we can write $g= g_1 + g_2$ and 
		$$
		\tilde g_i (z) = g_i (P_1+ Mz), \quad i=1,2.
		$$
		Arguing as in the proof of Proposition \ref{even-ness}, we have that a function of the form
		$$
		z \to \sum_{j=1}^N f ( M_j^{-1} (x-P_j)), \quad x=P_1+Mz
		$$
		for a fixed profile $f$, is even in the $z_2$ variable. Hence
		$$
		z \to  \sum_{j=1}^N \, 
		\,    E(M_j^{-1}(x-P_j)) , \quad x=P_1 + Mz$$
		is even with respect to $z_2$. Since
		$|x|^2 = R^2 + 2 {R h \over \sqrt{h^2 + R^2}}  z_1 + {h^2 \over h^2 + R^2} z_1^2 + z_2^2 ,$ then also $z \to \eta_0 (|P_1 + Mz|)$ is even in the $z_2$ variable. This proves that $\tilde g_1 (z_1 , z_2) = \tilde g_1 (z_1, - z_2).$

		Consider now $\tilde g_2$, and define $\tilde \eta_0 (z) = \tilde \eta_0 (|P_j + M_j z|).$ Arguing as in Proposition \ref{24} and using the radial symmetry of $\eta_0$, we obtain that
		\begin{align*}
			[ L( \eta_{0} \psi_j ) - \eta_0 L(\psi_j) ] (x) = [L_0 (\tilde \eta_0 \Psi_{\ve \mu} ) - \tilde \eta_0 L_0 (\Psi_{\ve \mu})] (M_j^{-1} (x-P_j)).
		\end{align*}
		Using the same argument as before, we conclude that $\tilde g_2$ is also even with respect to $z_2$. This concludes the proof.
		
	\end{proof}

	\section{}\label{AppeB}

	This section is to prove Proposition \ref{prop1}. 
	For this purpose, we recall a result contained in Lemma 6.1 in \cite{ddmw2}. To state that result, we need to introduce some notations.
	For numbers $m>2$, $\beta \in (0,1)$ and functions $h :\R^2 \to \R$, with $\|h\|_{m}$ and $\|h \|_{m,\beta} $ we denote respectively
	\begin{align*}
		\|h \|_{m} &=  \sup_{y\in  \R^2 } | (1+|y|)^{m} h(y)|] ,\\  
		\|h \|_{m,\beta} &=  \sup_{y\in \R^2 } (1+|y|)^{m+\beta}[ h]_{B(y,1)}   +     \|h \|_{m}.
	\end{align*}

	\medskip
	We have
	
	\begin{lemma}\label{lemat}
		(Lemma 6.1 in \cite{ddmw2}). Given $m>2$ and $0<\beta < 1$,
		there exist a constant $C>0$ and a solution $\phi =  \TT [ h]$ of problem 
		$$
		\Delta \phi +e^\Gamma\phi+h(y)=0\quad \mbox{in}\quad \RR^2
		$$
		for each $h$ with $\|h\|_{m} <+\infty$ that defines a linear operator of $h$ and satisfies the estimate
		$$\begin{aligned}
			& (1+|y|) | D_y \phi (y)|  +  | \phi (y)| \\  &
			\,  \le  \,  C \big [ \,  \log (2+|y|) \,\big|\int_{\R^2} h Z_0\big|  +    (1+|y|) \sum_{j=1}^2 \big|\int_{\R^2} h Z_j\big| 
			+  (1+|y|)^{-(m-2)} \|h\|_{m}   \,\big ]. \end{aligned}
		$$
		In addition, if
		$\|h\|_{m,\beta} <+\infty$, we have
		$$\begin{aligned}
			& (1+|y|^{2+\beta})  [D^2_y \phi]_{B_1(y),\beta}  +(1+|y|^2)  |D^2_y \phi (y)| \\  &
			\,  \le  \,  C \big [ \,  \log (2+|y|) \,\big|\int_{\R^2} h Z_0\big|  +    (1+|y|) \sum_{j=1}^2 \big|\int_{\R^2} h Z_j\big| 
			+  (1+|y|)^{-(m-2)} \|h\|_{m,\beta}   \,\big ]. \end{aligned}
		$$
		
	\end{lemma}
	
	This result is key for the following.
	\medskip
	
	\begin{proof}[Proof of Proposition \ref{prop1}]
		We consider a standard linear extension operator $h\mapsto \ttt h $ to entire $\R^2$,
		in such a way that the support of $\ttt h$ is contained in $B_{2\rho}$ and $\|\ttt h\|_{m,\beta} \le C\|h\|_{m,\beta, B_\rho}$ with $C$ independent of all large $\rho$. The operator $B_\ve$ is written in \eqref{By} and $b_\ve$ is defined in \eqref{defbe} and the coefficients are of class $C^1$ in entire $\R^2$ and  have compact support in $B_{2\rho}$. 
		Then we consider the auxiliary problem in entire space
		\be
		\Delta \phi  +  e^{\Gamma} \phi  + B_\ve[\phi] + b_\ve(y)\phi + \ttt h(y)  = \sum_{j=0}^2  d_{j} e^{\Gamma} Z_j \inn \R^2
		\label{001} \ee
		where, assuming that $\|h\|_{m}<+\infty$ and $\phi$ is of class $C^2$,
		$d_{j} = d_{j}[h,\phi]$ are the scalars so that
		\[
		d_j
		=\gamma_j \int_{\R^2} (B_\ve[\phi] +b_\ve(y)\phi + \ttt h(y))Z_j,  \quad \gamma_j^{-1} = \int_{\R^2}  e^{\Gamma} Z_j^2\quad \mbox{for}\quad j=0,1,2.
		\]
		Note that $b_\ve(y)$ is even in $y_2$, see \eqref{beven} and since $\phi$ is even in $y_2$, then $B_\ve[\phi]=\ve^2\mu^2B(\ve\mu y)[\phi]$ is also even in $y_2$ by the form of $B$ in \eqref{Bexpr}, therefore using \eqref{even}, we have
		\[
		d_2
		=\gamma_2 \int_{\R^2} (B_\ve[\phi] +b_\ve(y)\phi + \ttt h(y))Z_2 =0.
		\]
		On the other hand  
		\[
		\int_{\R^2} B_\ve[\phi] Z_j = O(\|\phi\|_{*,m-2,\beta}) \frac{\ve\mu}{\sqrt{|\log \ve|}},\quad 
		\int_{\R^2} b_\ve(y)\phi Z_j = O(\|\phi\|_{*,m-2,\beta}) \ve\mu\sqrt{|\log \ve|} 
		\]
		for $j=0,1$, where
		\[
		\|\phi\|_{*, m-2,\beta} =  \|  D^2_y\phi \|_{m,\beta}   + \|  D_y\phi \|_{m-1}+ \|\phi \|_{m-2}.
		\]
		Consequently, we get
		\[
		\int_{\R^2} [B_\ve[\phi]+b_\ve(y)\phi]Z_j  
		= O(\|\phi\|_{*,m-2,\beta}) \ve\mu \sqrt{|\log \ve|}=O(\|\phi\|_{*,m-2,\beta})\rho^{-1}.
		\]
		On the other hand
		\[
		\int_{\R^2\setminus B_\rho} h(y)Z_0 = O(\rho^{-(m-2)})\|h\|_{m,\beta,B_\rho},\:\:
		\int_{\R^2\setminus B_\rho} h(y)Z_1 = O(\rho^{-(m-1)})\|h\|_{m,\beta,B_\rho}
		\]
		In addition, we can compute that
		\[ 
		\| B_\ve[\phi] + b_\ve(y)\phi\|_{m,\beta}  \le  C \frac{\delta}{|\log\ve|}  \|\phi\|_{*,m-2,\beta},
		\]
		where the dominant term is given by $B_\ve[\phi]$.
		
		Let us consider the Banach space $X$ of all $C^{2,\beta}(\R^2)$
		functions with
		$
		\|\phi\|_{*, m-2,\beta} <+\infty.
		$
		We find a solution of \equ{001}
		if we solve the equation
		\be\label{fp}
		\phi  =   \mathcal A  [\phi]  +  \mathcal H ,\quad \phi \in X
		\ee
		where
		\[
		\mathcal A  [\phi]
		=  \mathcal T\Big [B_\ve[\phi]+ b_\ve(y)\phi -\sum_{j=0}^1 d_{j}[0, \phi] e^{\Gamma}Z_j  \Big] ,\quad
		\mathcal H  = \mathcal T\Big [  \ttt h   - \sum_{j=0}^1 d_{j}[\ttt h,0] e^{\Gamma}Z_j  \Big] .
		\]
		and $\mathcal T$ is the operator built in Lemma  \ref{lemat}.
		We observe that
		\[
		\|\mathcal A  [\phi] \|_{*, m-2,\beta} \le C\frac{\delta}{|\log\ve|}  \|\phi\| _{*, m-2,\beta}, \quad \|\mathcal H \|_{*, m-2,\beta} \le C \|h \|_{ m,\beta, B_\rho}.
		\]
		So we find that Equation \equ{fp} has a unique solution, that defines a linear operator of $h$, and satisfies
		\[
		\|\phi \|_{*, m-2,\beta} \ \le\  C \|h \|_{m,\beta, B_\rho}
		\]
		The result of the proposition follows by just setting $T[h] = \phi\big|_{B_\rho}$. The proof is concluded. 
		
	\end{proof}


\end{document}